\documentclass[]{imsart}
\RequirePackage{amsthm,amsmath,amsfonts,amssymb}
\usepackage[linesnumbered,ruled,vlined]{algorithm2e}
\RequirePackage[numbers,sort&compress]{natbib}
\RequirePackage[colorlinks,citecolor=blue,urlcolor=blue]{hyperref}
\RequirePackage{graphicx}
\usepackage[margin=1.2in]{geometry}  
\startlocaldefs

\theoremstyle{plain}

\theoremstyle{remark}

\renewcommand{\cite}{\citet*}
\newcommand{\bX}{\mathbf{X}}
\newcommand{\bx}{\mathbf{x}}
\newcommand{\bu}{\mathbf{u}}

\newcommand{\bY}{\mathbf{Y}}
\newcommand{\by}{\mathbf{y}}
\newcommand{\bF}{\mathbf{F}}
\newcommand{\bD}{\mathbf{D}}

\newcommand{\bG}{\mathbf{G}}
\newcommand{\tbG}{\tilde\bG}

\newcommand{\be}{\mathbf{e}}
\newcommand{\br}{\mathbf{r}}
\newcommand{\bmu}{\mathbf{\mu}}
\newcommand{\mb}{\mathbf}
\newcommand{\cov}{\mathop{\mathrm{Cov}}}

\newcommand{\bGa}{\mb{\Gamma}}
\newcommand{\bSi}{\mb{\Sigma}}

\newcommand{\bA}{{\bf A}}
\newcommand{\bB}{{\bf B}}
\newcommand{\bw}{{\bf w}}
\newcommand{\bU}{{\bf U}}
\newcommand{\bV}{{\bf V}}
\newcommand{\bC}{{\bf C}}

\newcommand{\bS}{{\bf S}}
\newcommand{\rtr}{{\rm tr}}

\newcommand{\bI}{{\bf I}}
\renewcommand{\(}{\left(}
\renewcommand{\)}{\right)}

\DeclareMathOperator{\tr}{tr}         
    
        \DeclareMathOperator{\E}{E}

\newtheorem{thm}{Theorem}[section]

\newtheorem{lem}{Lemma}[section]
\newtheorem{rem}{Remark}[section]

\allowdisplaybreaks

\endlocaldefs

\begin{document}

\begin{frontmatter}
\title{Limiting behavior of bilinear forms for the resolvent of sample covariance matrices under elliptical distribution with applications}
\runtitle{Bilinear forms for the resolvent of sample covariance matrices}

\begin{aug}

\author[A]{\fnms{Yanqing} \snm{ Yin}
			\ead[label=e3]{yinyq@cqu.edu.cn}}			
			\and
\author[B]{\fnms{Wang} \snm{ Zhou}
			\ead[label=e4]{wangzhou@nus.edu.sg}}
			
\runauthor{Yin and Zhou.}

\address[A]{School of Mathematics and Statistics,
			Chongqing University\printead[presep={,\ }]{e3}}

\address[B]{Department of Statistics and Data Science,  National University of Singapore \printead[presep={,\ }]{e4}}
\end{aug}

\begin{abstract}
In this paper, we introduce a joint central limit theorem (CLT) for specific bilinear forms, encompassing the resolvent of the sample covariance matrix under an elliptical distribution. Through an exhaustive exploration of our theoretical findings, we unveil a phase transition in the limiting parameters that relies on the moments of the random radius in our derived CLT. Subsequently, we employ the established CLT to address two statistical challenges under elliptical distribution. The first task involves deriving the CLT for eigenvector statistics of the sample covariance matrix. The second task aims to ascertain the limiting properties of the spiked sample eigenvalues under a general spiked model. As a byproduct, we discover that the eigenmatrix of the sample covariance matrix under a light-tailed elliptical distribution satisfies the necessary conditions for asymptotic Haar, thereby extending the Haar conjecture to broader distributions.
\end{abstract}

\begin{keyword}[class=MSC]
\kwd[Primary ]{62H15}
\kwd{62B20}
\kwd[; secondary ]{62D10}
\end{keyword}

\begin{keyword}
\kwd{high dimension
  covariance matrix; central limit theorem; Haar conjecture; elliptical distribution; spiked model}
\end{keyword}

\end{frontmatter}

\section{Introduction and motivation}

The covariance matrix assumes a central role in multivariate statistical analysis, and numerous statistical inferences hinge on the spectral properties of the population covariance matrix (PCM). In high-dimensional scenarios, the sample covariance matrix (SCM) ceases to be a reliable estimator for the PCM in a spectral sense. Nevertheless, considerable efforts have been devoted to exploring the relationship between them. This investigation is crucial as it aids in making statistical inferences based on the observed data.Starting from the pioneering work of \cite{Wishart28G}, which examined the properties of the eigenspace of Wishart matrices, researchers have devoted significant efforts to enhancing the generality of data models to better align with real-world applications. In the past decades, the most extensively investigated data model is undoubtedly the independent component structure (ICS). We refer the readers to \cite{BaiS10S, KnowlesY17A, BloemendalE14I, Bao2022} and references therein. This model, serving as a natural extension of the multivariate Gaussian, assumes that a high-dimensional random population is a linear transformation of a random vector with independent and identically distributed (i.i.d.) entries. However, it has been recognized that this model excludes some significant distribution families, such as the elliptical family. 

We define a random vector $\mathbf{y}$ to follow an {\bf elliptically correlated structure} (ECS) if and only if it has a stochastic representation given by:
\begin{align*}
  \by=\rho \bGa \bu+\bmu.
\end{align*}
Here, the matrix $\mathbf{\Gamma} \in \mathbb{R}^{p \times p}$ and vector $\mathbf{\mu} \in \mathbb{R}^{p}$ are non-random, with ${\rm rank}(\mathbf{\Gamma}) = p$. The scalar variable $\rho \geq 0$ represents the radius of $\mathbf{y}$, and $\mathbf{u} \in \mathbb{R}^p$ is the random direction. The random direction $\mathbf{u}$ is independent of $\rho$ and uniformly distributed on the unit sphere $S^{p-1}$ in $\mathbb{R}^p$, denoted by $\mathbf{u} \sim U(S^{p-1})$ in the subsequent discussion.
This data model naturally extends the concept of multivariate normal distributions, offering a distinct orientation compared to the independent component structure (ICS) model. It adeptly captures the dependence structure among multiple variables, providing a versatile framework for analyzing complex multivariate data. Specifically, when $\rho^2 \sim \chi^2(p)$, the resulting distribution aligns with the multivariate Gaussian distribution. Elliptical distributions are well-suited for modeling the inherent structure of real-world datasets across diverse fields such as finance, biology, and engineering. Among the commonly employed distributions within this category are the multivariate Student's t-distribution, multivariate Cauchy distribution, and elliptical Gamma distribution.

In recent years, there has been a concerted effort to explore the spectral properties of the sample covariance matrix under elliptical distribution using random matrix theory (RMT). For further insights, interested readers are encouraged to explore works such as \cite{Karoui09C, Hu2019, Hu2019a, LiYin2023TAMS}. Nevertheless, numerous spectral properties crucial for statistical inferences remain unexplored. In light of this, our efforts are directed towards addressing this gap. More specifically, we aim to establish a joint CLT for several bilinear forms involving the resolvent of the sample covariance matrix under ECS. Following the establishment of the newly developed CLT, we employ it to systematically investigate two aspects. Firstly, we delve into the asymptotic properties of the eigenvectors of the sample covariance matrix. Secondly, we turn our attention to the spiked sample eigenvalues within the framework of a comprehensive spiked model. 

The primary contributions of our work are outlined as follows.
\begin{itemize}
	\item [1.] We derived a novel joint CLT for bilinear forms that incorporate the resolvent of the sample covariance matrix under the ECS scenario. This accomplishment allows for the exploration of the interdependence structure among the entries of the resolvent matrix and sheds light on the interdependence structure among their linear combinations. 
	\item [2.] Through the application of the newly established CLT, we have derived a corresponding CLT for eigenvector statistics of the sample covariance matrix under ECS model. Diverging from linear eigenvalue statistics, we have uncovered a noteworthy phase transition regarding the dependence of the CLT on the fourth moment of the radius variable $\rho$. Specifically, when the asymptotic variance of the squared radius $\rho^2$ vanishes, the CLT becomes independent of the specific underlying distribution. Conversely, the presence of non-vanishing asymptotic variance leads to a discernible dependence of the CLT on the specific underlying distribution. This phenomenon illuminates the extension of the Haar conjecture regarding sample covariance matrices to encompass all light-tailed elliptical distributions.
	\item [3.] We establish a connection between bilinear forms and random matrices that govern the asymptotic behaviors of spiked sample eigenvalues in a general spiked model. Consequently, we achieve the CLT for spiked sample eigenvalues under ECS model. Again, in the case of a light-tailed elliptical distribution, the behavior of spiked sample eigenvalues mirrors that observed under the Gaussian case. This observation underscores the universality phenomenon inherent in applying Principal Component Analysis (PCA) across all light-tailed elliptical distributions.
\end{itemize}

The subsequent sections of this paper are structured as follows. In the upcoming section, we will provide essential background results to facilitate a comprehensive understanding of our findings. Following that, we will present our main results. Moving on to the third section, we delve into the  application of our CLT within the context of two specific statistical problems as mentioned before. The detailed proofs will be deferred to the appendix, following a concluding discussion.

Throughout this paper, we represent the spectral norm of a matrix by $|\cdot|$. The symbol $C$ denotes a constant that may assume different values depending on the context. For any real sequences $a_n$ and $b_n$, we use $a_n=o(b_n)$ to express the relationship $a_n/b_n\to 0$ as $b_n\to \infty$ and $a_n=O(b_n)$ to indicate that $a_n/b_n\leq C$ as $b_n\to \infty$. Throughout this paper, $\be_j$ represnts the $j$-th column of the identity matrix.

\section{Prior definitions and main results}
This section aims to present our primary theoretical results. We initiate this discussion by introducing relevant definitions and outlining our model assumptions. 
\subsection{Definitions and model assumptions.}
Let $\bA_n$ represent a $p \times p$ symmetric matrix with eigenvalues {$\lambda_1\le\ldots\le\lambda_p$}. The \emph{\bf Empirical Spectral Distribution} (ESD) of $\mb A_n$ is defined as
 $
   F^{\bA_n}(x)=\frac{1}{p}\sum_{j=1}^p  I(\lambda_j\leq x),
$
where $I(\cdot)$ is the indicator function.
If, in the limit as $p$ and $n$ approach infinity, the limit of $F^{\bA_n}(x)$ exists, it is termed the \emph{\bf Limit Spectral Distribution} (LSD). A crucial tool in RMT for exploring the spectral properties of $\bA_n$ is the \emph{\bf Stieltjes transform}, denoted by
	$$m_{F^{\bA_n}}(z)=\int\frac{1}{y-z}d{F^{\bA_n}}(y),~~\mbox{ where }
	z\in\mathbb{C}^{+}\equiv\{z=u+iv\in\mathbb{C}:v>0\}.$$
It is evident that the Stieltjes transform is linked to the \emph{\bf resolvent} $$\Upsilon\(\bA_n,z\)\doteq\(\bA_n-z\)^{-1},$$
as expressed by  $m_{F^{\bA_n}}(z)=\frac1p\tr\Upsilon\(\bA_n,z\).$

The primary model of interest in this paper is the sample covariance matrix under elliptical distributions, expressed as $$\bS_{0,n}=\frac{1}n\bGa_n\bX_n\bX_n^T\bGa_n^T.$$  We proceed to enumerate the assumptions as follows.\begin{description}
	\item \emph{Assumption (a) [ECS]}: The columns of $\bX_n$ follow an elliptical distribution, represented as $\bx_j=\rho_j{\bf u_j}$, where $1\leq j\leq n$. Here, the random radius $\rho_j$'s are independent and identically distributed (i.i.d.) random variables with $\E(\rho_1^2)=p$ and  $\E\(\rho_1^4\)=m_p=\nu_p+p^2\geq p^2$ and the directions ${\bf u}_j\overset{i.i.d.}{\sim} U(S^{p-1})$; 
	\item \emph{Assumption (b) [Bounded norms]}: $\bGa_{n}$ is a $p\times p$ non-random matrix with a uniformly bounded spectral norm. As $n\to\infty$, the ESD of $\boldsymbol\Sigma_n=\bGa_{n}\bGa_{n}^T$ denoted by $H_{1n}$ converges weakly to a proper distribution $H_{1}$. Furthermore, the distribution function $H_{2n}$ of $m_p^{-1/2}{\rho_1^2}$ converges to $H_{2}$ whose support is bounded. 
	\item \emph{Assumption (c) [High dimensional framework]}: The dimension to sample size ratio $c_n=p/n\to c\in(0,\infty)$ as $n\to\infty$.
\end{description}

The ECS model under Assumption (a) encompasses a broad range of elliptical distributions. It is notable that the variance of $\rho^2$, denoted as $\nu_p$, can assume any order of $p$. However, it is conceivable that when $\nu_p/p^2\to\infty$, the spectral properties of $\bS_{0,n}$ will be primarily determined by the distribution of $\rho$. Thus, it is reasonable to consider a normalization and shift focus to the normalized sample covariance matrix 
$$\bS_{n}=\frac{\sqrt{{p^2}/m_p}}n\bGa_n\bX_n\bX_n^T\bGa_n^T=\sqrt{{p^2}/m_p}\bS_{0,n}.$$
Leveraging this normalization, we can explore a more general model than in \cite{Hu2019, Hu2019a, LiYin2023TAMS}, where the LSD and CLT for linear spectral distributions (LSS) are considered and than in \cite{Wentwell2022} where the distribution of largest eigenvalue are considered. Evidently, our model encompasses their models as special cases when $\nu_p=O(p)$.
Assumptions (b) and (c) are commonly employed in RMT.

\subsection{First order limit : the LSD}\label{secct}
As a foundational step, we first introduce the results related to the LSD of $\bS_n$. 
\begin{thm}
	Suppose the Assumptions $(a)-(c)$ hold. With probability $1$, as $n\to\infty$, the ESD of $\bS_n$ converges weakly to a non-random probability distribution function $F^{c,H_{1},H_{2}}$. To be specific, we have 
	\begin{description}
		\item [I. Trivial scenarios]: If $H_{1}=1_{[0,\infty)}$ or $H_{2}=1_{[0,\infty)}$, then $F^{c,H_{1},H_{2}}=1_{[0,\infty)};$
		\item [II. Non trivial scenarios]: If $H_{1}\neq 1_{[0,\infty)}$ and $H_{2}\neq 1_{[0,\infty)}$, for each $z\in \mathbb{C}^+$,
		\begin{align}\label{mpf}
			\begin{cases}
				m(z)=-z^{-1}(1-c^{-1})-z^{-1}c^{-1}\int\frac1{1+q_1(z)y}dH_{2}(y)\\
				m(z)=-z^{-1}\int\frac1{1+q_2(z)x}dH_{1}(x)\\
				m(z)=-z^{-1}-c^{-1}q_1(z)q_2(z)
			\end{cases}
		\end{align}
		is viewed as a system of equations for the complex vector $\left(m(z),q_1(z),q_2(z)\right)$, then (\ref{mpf}) has a unique solution in the set
		$$U=\left\{\left(m(z),q_1(z),q_2(z)\right):\Im m(z)>0,\Im(zq_1(z))>0,\Im q_2(z)>0 \right\}.$$ Also, the stieltjes transform of $F^{c,H_{1},H_{2}}$, denoted by $m_F(z)$, together with the other two functions $g_1(z)$ and $g_2(z)$, both of which are analytic on $\mathbb{C}^+$, are given by this solution.
	\end{description}
\end{thm}
This theorem represents a minor extension of Theorem 1 in \cite{Hu2019a}, and as such, we omit its proof.
In general, the covariance matrix under an elliptical distribution typically belongs to the domain of the so-called separable sample covariance matrix model. For light-tailed elliptical distributions where $\nu_p=O(p),$ it is observed that the distribution of $\rho^2/p$ becomes degenerate, leading to $H_2=1_{[1,\infty)}$. This degeneracy results in the system of equations reducing to the single M-P equation.

For future reference, we introduce here the {\bf companion} of $\bS_n$, defined as $\underline\bS_n=\frac{\sqrt{{p^2}/m_p}}n\bX_n^T\bGa_n^T\bGa_n\bX_n.$ Note that the spectra of $\bS_n$ and $\underline\bS_n$ only
differ by $|p-n|$ zero eigenvalues.  It follows that
\begin{align*}
  F^{\underline\bS_n}(x)=(1-c_n)I_{[0,\infty)}+c_nF^{\bS_n}(x),
\end{align*}
from which we get
\begin{align}\label{sx}
  \quad \underline F(x)=(1-c)I_{[0,\infty)}+c F(x), \ m_{F^{\underline\bS_n}}(z)=-\frac{1-c_n}z+c_nm_{F^{\bS_n}}(z),\ z\in \mathbb{C}^+,
\end{align}
and as $n\to\infty$
\begin{align}\label{sxl}
  \underline m(z):=\underline m_{F^{c,H_1,H_2}}(z)=-\frac{1-c}z+cm(z), \ z\in \mathbb{C}^+.
\end{align}
Therefore, by comparing (\ref{mpf}) and (\ref{sxl}), we can establish the following relationship
\begin{align}\label{g12}
	\begin{cases}
		zg_1(z)=-c\int\frac x{1+g_2(z)x}dH_1(x),\\
		zg_2(z)=-\int\frac y{1+g_1(z)y}dH_2(y).
	\end{cases}
\end{align}

\subsection{Bilinear forms for resolvent} \label{secthm}
After the discussion presented in the last subsection, the LSD of $\bS_n$ is clarified under specific scenarios. Regarding the second-order properties of the SCM under ECS, previous research has made significant efforts in this area in recent years. In the case where $\nu_p=O(p)$, the CLT for LSS was considered in \cite{Hu2019} and \cite{LiYin2023TAMS}. Their CLT reveals the dependence of the limiting mean and variance on both $H_1$ and the variance of $\rho^2$ ($\nu_p$). This dependence, distinct from the results in the ICS case, indicates the influence of the nonlinear dependence of variables in the multivariate population. In \cite{Hu2019a}, the authors establish the CLT for LSS under the case where $\nu_p=O(p^2)$, demonstrating a different convergence rate of $\sqrt{n}$ in the corresponding CLT. 

At the heart of the spectral decomposition of SCM lies the eigenmatrix, an ensemble of eigenvectors that unveils intricate patterns within the data. Understanding its behavior under the influence of high dimensionality, varying correlation structures, and non-normality is essential for developing robust methodologies tailored to modern data challenges. To the best of our knowledge, no prior research has specifically explored the properties of eigenvectors in the context of the SCM and the spiked model under ECS. Our study endeavors to bridge these gaps in the existing literature.
We will achieve this objective by establishing a joint CLT for processes of bilinear forms related to the resolvent $\Upsilon\(\bA_n,z\).$ More specifically, we consider $r$ correlated bilinear forms: $$\Bigg\{\mathcal{B}_{\mathfrak{r}}\(\Upsilon\(\bA_n,z\)\)=\mb\pi_{n,2\mathfrak{r}-1}^T\Upsilon\(\bA_n,z\)\mb\pi_{n,2\mathfrak{r}}\Bigg\}_{\mathfrak{r}=1}^r,$$
where the $p$-dimensional non-random vectors $\mb\pi_{n1},\cdots,\mb\pi_{n{2\mathfrak{r}}}$ are assumed to have unit norms without loss of generality. Our focus is on the multivariate process equipped with $z.$ We note that in this paper, our consideration is limited to $z$ values in a proper region denoted as $\mathcal{Z}$, ensuring that both the resolvent and $\(\bI_p+g_{2n}^0(z)\bSi_n\)^{-1}$ exist and are bounded in spectral norm with high probability for the given scenario. 

 The following theorem establishes the convergence of a single bilinear form $\mathcal{B}_{\mathfrak{r}}\(\Upsilon\(\bA_n,z\)\).$
\begin{thm}\label{thmlsd}
Under Assumptions (a-c), the following conclusion holds for any $z\in\mathcal{Z}$,
$$\mathcal{B}_{\mathfrak{r}}\(\Upsilon\(\bA_n,z\)\)+z^{-1}\mb\pi_{n,2\mathfrak{r}-1}^T
\(\bI_p+g_{2n}^0(z)\bSi_n\)^{-1}\mb\pi_{n,2\mathfrak{r}}\to0, \quad a.s., \quad\mathfrak{r}=1,\cdots,r.$$
 Here 
$\(\underline{m}_{F^{c_n,
		H_{n1},H_{n2}}}(z),g_{1n}^0(z),g_{2n}^0(z)\)$ is defined in (\ref{mpf}) by
replacing $c,H_1,H_2$ with\\ $c_n,H_{1n},H_{2n}$.
\end{thm}
We are now ready to establish the convergence of the multivariate process $$\(\mathcal{B}_{1}\(\Upsilon\(\bA_n,z\)\),\cdots,\mathcal{B}_{r}\(\Upsilon\(\bA_n,z\)\)\)^T.$$
We will present the case $r=2$ and the general cases are similar and therefore omitted.
Write
\begin{align*}
	M_{n}(z)=\begin{pmatrix}
		M_{n1}(z)\\
		M_{n2}(z)
	\end{pmatrix}=\begin{pmatrix}
		\sqrt p\left(\mathcal{B}_{1}\(\Upsilon\(\bA_n,z\)\)+z^{-1}\pi_{n1}^T
		\(\bI_p+g_{2n}^0(z)\bSi_n\)^{-1}\pi_{n2}
		\right)\\
		\sqrt p\left(\mathcal{B}_{2}\(\Upsilon\(\bA_n,z\)\)+z^{-1}\pi_{n3}^T
		\(\bI_p+g_{2n}^0(z)\bSi_n\)^{-1}\pi_{n4}\right)
	\end{pmatrix}.
\end{align*}
\begin{thm}\label{thclt}
Under Assumptions (a-c), define
\begin{align*}
	&r_{jk}(z_1,z_2)=\lim_{n\to+\infty}\pi_{nj}^T\(\bI_p+g_{2n}^0(z_1)\bSi_n\)^{-1}\bSi_n\(\bI_p+g_{2n}^0(z_2)\bSi_n\)^{-1}\pi_{nk},\\
	& r_{jk}(z)=\lim_{n\to+\infty}\pi_{nj}^T\(\bI_p+g_{2n}^0(z)\bSi_n\)^{-2}\bSi_n\pi_{nk},\quad j,k \in \{1,2,3,4\}.
\end{align*}
We have the two dimensional process $M_n(z)$ for $z\in\mathcal{Z}$
  converges weakly to a two dimensional zero-mean Gaussian process $M(z)$ with a covariance function  
\begin{align*}
	\cov\(M_1(z_1),M_2(z_2)\)=&h_{1}(z_1,z_2)r_{14}(z_1,z_2)r_{23}(z_1,z_2)+h_{1}(z_1,z_2)r_{13}(z_1,z_2)r_{24}(z_1,z_2)\\
&+{h_{2}(z_1,z_2)}r_{12}(z_1)r_{34}(z_2).
\end{align*} 
Here  

\begin{align*}
	h_1(z_1,z_2)=&\frac {c\(z_1g_2(z_1)-z_2g_2(z_2)\)}{z_1^2z_2^2\(g_1(z_1)-g_1(z_2)\)\(1-d(z_1,z_2)\)},\\
	h_2(z_1,z_2)=&\frac{cg_2'(z_1)g_2'(z_2)\(\underline m(z_1)g_2(z_2)-\underline m(z_2)g_2(z_1)\)}{g_2(z_1)g_2(z_2)\(g_1(z_1)-g_1(z_2)\)},\\
 d(z_1,z_2)=	&\frac1{z_1z_2}\frac{z_1g_1(z_1)-z_2g_1(z_2)}{g_1(z_1)-g_1(z_2)}\frac{z_1g_2(z_1)-z_2g_2(z_2)}{g_2(z_1)-g_2(z_2)}.
\end{align*}
\end{thm}

With the aid of the above theorem, one can justify the limiting joint distribution of certain variables related to the resolvent. For instance, by choosing $$\mb\pi_{n1}=\be_j, \ \mb\pi_{n2}=\be_k, \ \mb\pi_{n3}=\be_l, \ \mb\pi_{n4}=\be_{t}$$
and letting $z_1\to z_2,$ we can derive the limiting variances and covariance of the entry lying in the $j$-th row, $k$-th column, and the entry lying in the $l$-th row, $t$-th column of $\sqrt{p}\Upsilon\(\bA_n,z\)$.

We would like to delve deeper into the theorem above. It is crucial to note that both the limiting covariance function and the centralizing term rely on the distribution of the radius $\rho$ through $g_{2n}^0(z)$. This dependence is solely influenced by the properties of $m_p^{-1/2}\rho_1^2$, whose distributions remain consistent when $\nu_p=o(p^2)$. Consequently, we can deduce that a phase transition will occur as $\nu_p$ transitions from $o(p^2)$ to the order of $p^2$. In other words, for a light-tailed elliptical distribution, the asymptotic properties of the bilinear forms $\mathcal{B}\(\Upsilon(\bA_n,z)\)$ are independent of specific distributions. However, when the fluctuation of $\rho^2$, denoted by $\nu_p/p^2$, deviates from 0, a critical point is reached, marking a shift in the scenario. At this juncture, the impact of nonlinear dependence, induced by the random radius, becomes pronounced enough to influence the asymptotic properties of $\mathcal{B}\(\Upsilon(\bA_n,z)\)$. 

\section{Statistical applications}

In this dedicated section, we leverage the robust CLT established for bilinear forms and channel its applicability into two distinct directions within the statistical domain. These directions not only broaden the scope of our theoretical framework but also enhance its practical relevance in addressing nuanced challenges encountered in statistical analyses.
The first avenue of exploration involves the functional CLT, a pivotal concept in the realm of eigenvector statistics pertaining to sample covariance matrices. Our established CLT for bilinear forms provides a solid foundation for delving into the intricacies of eigenvector statistics. 
Simultaneously, our focus extends to the second direction, which centers around the analysis of spiked eigenvalues and eigenvectors within a spiked model. 

\subsection{Functional CLT for eigenvector statistics}\label{FCLT} In this subsection, we embark on the application of the theoretical insights acquired in the preceding section to scrutinize the asymptotic properties of the eigenmatrix of $\bS_n.$ Before delving into the details, we find it necessary to introduce some fundamental definitions and background information that will lay the groundwork for our subsequent analysis.

	Given $\pi_n$, the \emph{\bf vector empirical spectral distribution} (VESD)
        function based on eigenvalues and eigenvectors of matrix $\mb
        A_n$ is defined as 	
	$$F_{v,\pi_n}^{\mb A_n}(x)=\sum_{j=1}^p|q_j|^2I(\lambda_j\leq x).$$ Understanding this function is pivotal for unraveling the intricacies of the eigenmatrix and its statistical behavior. Now, when the underlying data originates from a multivariate Gaussian distribution, the sample covariance matrix $\bS_n$ is a Wishart matrix. According to established results like those in \cite{Anderson03I} or Corollary 2.2 of \cite{Dumitriu2002}, the eigenmatrix of $\bS_n$ follows the Haar distribution. Formally, if we express $\bS_n$ as $\mb U_n \Lambda_n \mb U_n^T$, where $\mb U_n$ is the orthogonal matrix containing the eigenvectors and $\Lambda_n$ is a diagonal matrix of eigenvalues, then $\mb U_n$ conforms to the uniform distribution over the group formed by all orthogonal matrices. 
Moreover, for any unit vector $\mb\pi_n\in\mathbb{R}^p$, the random vector $\mb q_n=\mb
U_n\pi_n\doteq\(q_1,\cdots,q_p\)^T$ follows a uniform distribution over the unit sphere. This underscores the uniformity and isotropy of the eigenvectors associated with the Wishart matrix. Furthermore, consider the stochastic process defined as
$$\mathbb
{Q}_p(t)=\sqrt{\frac{p}{2}}\sum_{j=1}^{[pt]}\(|q_j|^2-\frac1n\)\overset{d}{=}\sqrt{\frac{p}{2}}\frac{1}{|\mb z|^2}\sum_{j=1}^{[pt]}\(|z_j|^2-\frac{|\mb z|^2}{p}\).$$ This process converges in distribution to a Brownian Bridge $\mathbb{B}(t)$ as $p\to \infty$, see Page 334 in \cite{BaiS10S}. This convergence provides a bridge to understanding the limiting behavior of the eigenmatrix. 
For any matrix $\mb A_n$, we introduce a time transformation denoted as $\mathcal {Q}_p^{\mb A_n}(x)=\mathbb {Q}_p(F^{\mb A_n}(x)).$ This transformation is applied to the stochastic process $\mathbb{Q}_p(t)$ using the ESD $F^{\mb A_n}(x)$. Subsequently, the transformed process $\mathcal{Q}_p^{{\bS}_n}(x)$ serves as an approximation to $\mathbb{B}(F^{c,H_1,H_2}(x))$.
Recalling the definitions of the ESD and the VESD, we can express the transformed process as follows: 
$$\mathcal {Q}_p^{\bS_n}(x)=\sqrt{\frac{p}{2}}\(F_{v,\pi_n}^{\bS_n}(x)-F^{\bS_n}(x)\).$$
This transformation allows us to reframe the study of $\mathbb{Q}_p(t)$ into the investigation of the discrepancy between the ESD and VESD. This shift in perspective not only simplifies the analysis but also provides a meaningful connection between the properties of the eigenmatrix and the convergence behavior encapsulated in $\mathbb{Q}_p(t)$. 

Through this investigatory approach, notable efforts have been dedicated to exploring the universality of the Haar conjecture in high dimension. A somewhat unexpected revelation, as articulated in Theorem 10.2 within \cite{BaiS10S}, asserts that, under ICS conditions, the eigenmatrix's requisite condition for adhering to the Haar conjecture demands that the underlying distribution exhibits a fourth moment akin to a Gaussian distribution. This phenomenon implies a substantial impact of the fourth moment on the asymptotic structure of the eigenmatrix under ICS.  
The impact of nonlinear elliptical correlation on the asymptotic structure of the eigenmatrix can be elucidated through the application of the newly established CLT for bilinear forms. To illustrate this, consider any function $g$ that is analytic on an open set containing the supports of $F_{v,\pi_n}^{\bS_n}(x)$ and $F^{\bS_n}(x)$. By the Cauchy integral formula, the following relation holds for large $n$:
$$
\int g(x) \mathrm{d} \(F_{v,\pi_n}^{\bS_n}(x)-F^{\bS_n}(x)\)=-\frac{1}{2 \pi \mathrm{i}} \int_{\mathcal{C}} g(z)\( s_{F_{v,\pi_n}^{\bS_n}}(z)-s_{F^{\bS_n}}(z)\) \mathrm{d} z,
$$
where $\mathcal{C}$ is a contour that encompasses the real interval defined by:
\begin{align}\label{inter}
	\Bigg[a\liminf_n\lambda_{min}^{\bSi_n}I_{(0,1)}(c)\(1-\sqrt{c}\)^2,b\ \limsup_n\lambda_{max}^{\bSi_n}\(1+\sqrt{c}\)^2\Bigg].
\end{align}
Here, $a$ and $b$ represent the lower and upper bounds of the support of $H_2$, respectively. 

Let's define $s_{c_n,\pi_n}^{\bSi_n,\nu_p}(z)=z^{-1}\mb\pi_{n}^T
\(\bI_p+g_{2n}^0(z)\bSi_n\)^{-1}\mb\pi_{n}$,  representing the Stieltjes transform of the anisotropic M-P law $F_{c_n,\pi_n}^{\bSi_n,\nu_p}(x)$. We shall introduce 
$$\mathbb{G}_n(x)=\sqrt{p}\(F_{{v},\pi_n}^{\bS_n}(x)-F_{c_n,\pi_n}^{\bSi_n,\nu_p}(x)\).$$
Consider test functions $\zeta_1,
 \cdots,\zeta_k$ analytic on an open set containing (\ref{inter}). The functional CLT for eigenvector statistics can then be expressed as follows: 
\begin{thm}\label{vecClt}
	Under the assumptions of Theorem \ref{thclt}, we have the following results.
\begin{description}
\item [I:] The $k$ dimensional random
  vectors $$\Psi_n=\(\psi_{1,n},\cdots,\psi_{k,n}\)'=\(\int \zeta_1(x)d
  \mathbb{G}_n(x), \cdots, \int \zeta_k(x)d \mathbb{G}_n(x)\)'$$ form a tight sequence.
\item [II:] The random vectors $\Psi_n$ converge weakly to a
  mean zero Gaussian vector $\Psi=\(\psi_{1},\cdots,\psi_{k}\)'$.
\item [III:] For $1\leq
  t,s\leq k,$
\begin{align}
	\cov\(\psi_{t},\psi_{s}\)=-\frac{1}{2\pi^2}\int_{\mathcal{C}_1}\int_{\mathcal{C}_2}\zeta_t(z_1)\zeta_s(z_2)\varpi(z_1,z_2)dz_1dz_2,
\end{align} where ${\cal C}_1, {\cal C}_2$ are two non-overlapping contours enclosing the support of $F^{c,H_1,H_2}$ and \begin{align*}
	\varpi(z_1,z_2)=&2h_{1}(z_1,z_2)r_{11}(z_1,z_2)r_{11}(z_1,z_2)+{h_{2}(z_1,z_2)}r_{11}(z_1)r_{11}(z_2).
\end{align*} 
\end{description}
\end{thm} 

This theorem unveils the universality of the Haar conjecture within the realm of elliptical distributions, even in the presence of nonlinear dependencies between variables, as long as $\nu_p=o(p^2)$.

\subsection{Asymptotic distribution of spiked eigenvalue and eigenvector} Gaining insights into the characteristics of spiked eigenvalues and their associated eigenvectors within a spiked sample covariance matrix holds paramount significance in a multitude of statistical applications, with a prominent example being Principal Component Analysis (PCA). In PCA, the identification of principal components associated with spiked eigenvalues serves as a pivotal mechanism for dimensionality reduction and feature extraction. This analytical approach proves particularly valuable in scenarios where datasets showcase a dominant signal. By pinpointing the spiked eigenvalues, one can extract crucial information about the intrinsic structure underlying the data. This understanding not only aids in optimizing data representation but also enhances the interpretability and effectiveness of statistical analyses.
Since the seminal work by \cite{Johnstone01D}, the exploration of this topic in the realm of high-dimensional statistics has garnered significant attention. Numerous authors have delved into the subject, progressively refining and expanding the models to accommodate a broader range of scenarios. The evolution of these models underscores the dynamic nature of statistical research in adapting to the demands of contemporary datasets. For the most recent advancements under the ICS model, we recommend consulting up-to-date references such as \cite{zhangzheng2022,Bao2022,TonyCai2020a,onatski2009testing}.

To better align the results of the spiked model with real high-dimensional datasets, we aim to explore the asymptotic distribution of sample spiked eigenvalues and eigenvectors under the ECS model by leveraging our result in Theorem \ref{thclt}. Consider the general spiked model, as introduced in \cite{zhangzheng2022}. Let $\bGa_n$ be decomposed using singular value decomposition as follows:
\begin{align*}
	\bGa_n=\bV\begin{pmatrix}
		\mb\Lambda_S^{1/2}&0\\
		0&\mb\Lambda_P^{1/2}
	\end{pmatrix}\bU^T
\end{align*}where $\bU$ and $\bV$ are orthogonal matrices, $\mb\Lambda_S$ is a diagonal matrix consisting of the spiked eigenvalues in descending order, and $\mb\Lambda_P$ is the diagonal matrix of the bounded non-spiked eigenvalues. Let's partition $\bU$ as $\bU=\(\bU_1,\bU_2\)$, where $\bU_1$ is a $p\times K$ submatrix of $\bU$. Define $\bY_n=\frac{\sqrt[4]{{p^2}/m_p}}{\sqrt n}\bX_n$ and
\begin{align*}
	\bSi_{1p}=\bU_2\mb\Lambda_P\bU_2^T=\bU\begin{pmatrix}
		\mb 0_S&0\\
		0&\mb\Lambda_P
	\end{pmatrix}\bU^T=\(\bU\begin{pmatrix}
	\mb 0_S^{1/2}&0\\
	0&\mb\Lambda_P^{1/2}
\end{pmatrix}\bU^T\)^2\triangleq\bG^2.
\end{align*}
Order the eigenvalues of $\bS_n$ as $\lambda_1\ge\lambda_2\ge\cdots\ge\lambda_p$. The sample spiked eigenvalues $\lambda_j (j=1,\cdots, K)$ of $\bS_n$ are determined by the equation involving the determinant:
\begin{align*}
	{\rm det}\left\{\mb\Lambda_S^{-1}-\bU_1^T\bY_n\(\lambda_j\bI-\bY_n^T\bSi_{1p}\bY_n\)^{-1}\bY_n^T\bU_1\right\}=0.
\end{align*}
Clearly, the columns of $\bU_1$ are orthogonal to $\bSi_{1p}$. It is noteworthy to highlight that, under ECS model, the target matrix mentioned above can be simplified by exploiting the property of elliptical distribution. Specifically, we have the option to diagonalize $\bSi_{1p}$ due to the characteristics of elliptical distributions. However, for the sake of maintaining generality and relevance to a broader spectrum of data models, we intentionally refrain from this simplification. 
In the literature under the ICS model, researchers have investigated the properties of the random matrix  $\bU_1^T\bY_n\(\lambda_j\bI-\bY_n^T\bSi_{1p}\bY_n\)^{-1}\bY_n^T\bU_1$ directly. For example, in \cite{Jiang2021b}, the authors established a general fourth-moment theorem to show that the distribution of this matrix remains the same when the underlying distribution is replaced by another one, provided they share the same fourth moment. In contrast, \cite{zhangzheng2022} studied the asymptotic distribution of the entries in this matrix directly by applying a martingale decomposition method. However, in this work, we will demonstrate through perturbation arguments that the study of this random matrix can be accomplished through properties of bilinear forms.  

 To see this, let $\tilde\bG=\bG+\varepsilon\bI_p$, so $\tilde\bG$ is invertible. 
Define 
\begin{align*}
	\Phi(z,\varepsilon)=&\bw_1^T\tbG^{-1}\(\tbG\bY_n\bY_n^T\tbG-z\bI_p\)^{-1}\tbG^{-1}\bw_1+z^{-1}\bw_1^T\tbG^{-1}\(\bI_p+g_{2n}^0(z)\tbG\tbG\)^{-1}\tbG^{-1}\bw_1,\\
	\Psi(z,\varepsilon)=&\bw_1^T\tbG^{-1}\tbG^{-1}\bw_1-\bw_1^T\tbG^{-1}\(\bI_p+g_{2n}^0(z)\tbG\tbG\)^{-1}\tbG^{-1}\bw_1.
\end{align*}
In the context of Theorem \ref{thclt}, considering the implications for $H_{1n}$ in connection to the ESD of $\bSi_{1p}$  and adjusting parameter definitions accordingly, we deduce that
$\sqrt p \Phi(z,\varepsilon)$ 
converges weakly to a Gaussian distribution  $N\(0,\sigma^2_1(z,\varepsilon)\)$, where
\begin{align*}
	\sigma_{1}^2(z,\varepsilon)=&\lim_{z_1\to z_2}\(2h_{1}(z_1,z_2)r_{11}^2(z_1,z_2)+{h_{2}(z_1,z_2)}r_{11}(z_1)r_{11}(z_2)\)\\
	=&\left[\frac {2c\(zg_2(z)\)'g_2'(z)}{z^2\(z\underline m(z)\)'}+\frac{c\(g_2'(z)\)^2\(\underline m(z)/g_2(z)\)'}{g_1'(z)}\right]\frac1{\(1+\varepsilon^2g_2(z)\)^2}.
\end{align*}
Using  the formula
\begin{align*}
	\bA\(\lambda\bI+\bB\bA\)^{-1}=	\(\lambda\bI+\bA\bB\)^{-1}\bA,
\end{align*}
and letting $\varepsilon\to0,$ we obtain 
\begin{align*}
	&z\bw_1^T\bY_n\(z\bI_p-\bY_n^T\bSi_{1p}\bY_n\)^{-1}\bY_n^T\bw_1\\
=&\lim_{\varepsilon\to0}\(-z\bw_1^T\tbG^{-1}\tbG^{-1}\bw_1-z^2\bw_1^T\tbG^{-1}\(\tbG\bY_n\bY_n^T\tbG-z\bI_p\)^{-1}\tbG^{-1}\bw_1\)\\
=&\lim_{\varepsilon\to0}\(-z\Psi(z,\varepsilon)-z^2\Phi(z,\varepsilon)\).
\end{align*}
Hence, define $$\mathcal{O}_{K\times K}(z)=\sqrt pz\(\bU_1^T\bY_n\(z\bI-\bY_n^T\bSi_{1p}\bY_n\)^{-1}\bY_n^T\bU_1+g_{2n}^0(z)\bI_{K}\),$$ and it follows that
\begin{align*}
	&\mathcal{O}_{11}(z)=\sqrt pz\(\bw_1^T\bY_n\(z\bI-\bY_n^T\bSi_{1p}\bY_n\)^{-1}\bY_n^T\bw_1+g_{2n}^0(z)\)
\end{align*}
converges weakly to Gaussian distribution $N\(0,\sigma^2_{11}(z)\)$ with variance 
\begin{align*}
	\sigma_{11}^2(z)=&\left[\frac {2cz^2\(zg_2(z)\)'g_2'(z)}{\(z\underline m(z)\)'}+\frac{cz^4\(g_2'(z)\)^2\(\underline m(z)/g_2(z)\)'}{g_1'(z)}\right].
\end{align*}
Applying a similar argument, we have  \begin{align*}
	&\mathcal{O}_{12}(z)=\sqrt pz\bw_1^T\bY_n\(z\bI-\bY_n^T\bSi_{1p}\bY_n\)^{-1}\bY_n^T\bw_2
\end{align*}
converges weakly to Gaussian distribution $N\(0,\sigma^2_{12}(z)\)$, where
\begin{align*}
		\sigma_{12}(z)=&\frac {cz^2\(zg_2(z)\)'g_2'(z)}{\(z\underline m(z)\)'}.
\end{align*}
Furthermore, we observe that ${\rm Cov}(\mathcal{O}{11}(z), \mathcal{O}{12}(z))$ converges to 0. By combining the aforementioned arguments, we essentially establish the following lemma.

\begin{lem}\label{thmspike}
Assuming the conditions outlined in Theorem \ref{thclt} are satisfied, we can establish the following conclusion: the random matrix $\mathcal{O}(z)$ weakly converges to a zero-mean Gaussian Orthogonal Ensemble (GOE) matrix $\mathcal{O}^{L}(z)=(\mathcal{O}^{L}{\mathfrak{i},\mathfrak{j}}(z))_{K\times K}$ with a covariance profile given by: 

 \begin{align*}
 	{\rm Cov}(\mathcal{O}^{L}_{\mathfrak{i},\mathfrak{j}}(z),\mathcal{O}^{L}_{\mathfrak{k},\mathfrak{l}}(z))=\begin{cases}
 		\sigma_{11}(z), \quad &  \mathfrak{i}=\mathfrak{j}=\mathfrak{k}=\mathfrak{l}\\
 		\sigma_{12}(z), \quad & (\mathfrak{i}=\mathfrak{k} \ {\rm{and}} \ \mathfrak{j}=\mathfrak{l}) \ {\rm or}\ (\mathfrak{i}=\mathfrak{l} \ {\rm{and}} \ \mathfrak{j}=\mathfrak{k});\\
 		0,\quad & {\rm otherwise},
 	\end{cases}
 \end{align*}    

where 
\begin{align*}
	\sigma_{1}^2(z)=&\frac {2cz^2\(zg_2(z)\)'g_2'(z)}{\(z\underline m(z)\)'}+\frac{cz^4\(g_2'(z)\)^2\(\underline m(z)/g_2(z)\)'}{g_1'(z)},\ 
	\sigma_{12}^2(z)=\frac {cz^2\(zg_2(z)\)'g_2'(z)}{\(z\underline m(z)\)'}.
\end{align*}
\end{lem}
This result provides a clear understanding of the asymptotic behavior of the random matrix $\mathcal{O}(z)$ under the specified conditions, connecting it to a GOE matrix with a well-defined covariance structure. Leveraging the aforementioned lemma and employing similar arguments as in \cite{zhangzheng2022,Jiang2021b}, we can derive the almost sure limit and limiting distribution of the spiked eigenvalues under ECS. More specifically, assuming that the population spiked eigenvalues of $\bSi_n$, denoted by $\alpha_1>\cdots>\alpha_K$, we obtain the following Theorem \ref{spike}. The proofs are very similar to theirs; therefore, we omit them to avoid repetition. We remind the reader to recall that, in the following, $g_{2n}^0(z)$ is associated with $\bSi_{1p}$. We also note that the multiple spiked eigenvalue case can be investigated similarly using Lemma \ref{thmspike}.  

\begin{thm}\label{spike}
	Under the assumptions in Theorem \ref{thclt}, further assuming the separation condition that $\min {j \neq k}\left|\frac{\alpha_k}{\alpha_j}-1\right|>d$, we have, for $k=1,\cdots,K$, $\Delta_{k}\doteq\frac{\lambda_k-\mathcal{G}_{2n}({\alpha_k})}{\mathcal{G}_{2n}({\alpha_k})}\to 0, a.s.,$ provided $\mathcal{G}_{2n}'(\alpha_k)>0$ where $\mathcal{G}_{2n}$ is the transition function that satisfies $g_{2n}^0(\mathcal{G}_{2n}\(z\))=-z^{-1}$.
Also, denoting $\theta_{k}=\mathcal{G}_{2n}(\alpha_k),$ we have 
	\begin{align*}
		\frac{\sqrt{n}\Delta_{k}}{\sigma_{\Delta_k}}\to N(0,1),
	\end{align*}
	where $\sigma_{\Delta_k}^2=\frac {2\(\theta_{k}g_2(\theta_{k})\)'}{\(\theta_{k}\underline m(\theta_{k})\)'g_{2}'(\theta_{k})\theta_{k}^2}+\frac{\(\underline m(\theta_{k})/g_2(\theta_{k})\)'}{g_1'(\theta_{k})}.$
\end{thm}

\begin{rem}
In the special case where $\nu_p=o(p^2),$ a particularly interesting insight emerges from our analysis. Leveraging the relationship $\underline m(z)=g_{2}(z),$ the transition function $\mathcal{G}_{2n}(\cdot)$ of spiked eigenvalues simplifies to a well-established form, precisely given by $\psi(z)=z+c z \int \frac{t}{z-t} d H_1(t),$ a result that aligns with existing knowledge in the field.
Furthermore, under this specific scenario, the variance of the standardized spiked eigenvalue  $\sigma_{\Delta_k}^2,$ simplifies to $\frac {2}{\underline m'(\theta_{k})\theta_{k}^2},$ which is consistent with the known result under Gaussian case. This remarkable finding underscores the robustness of the asymptotic properties of sample spiked eigenvalues across a diverse range of elliptical distributions, provided that $\nu_p=o(p^2).$ It highlights a certain universality in the behavior of these eigenvalues, irrespective of the specific characteristics of the elliptical distribution, offering valuable insights into their statistical properties in high-dimensional settings.
\end{rem}

Moving forward, let's examine the scenario of the spiked sample eigenvector, focusing initially on a simplified case. In this simplified setting, we operate under the assumption of a single population spiked eigenvalue, allowing us to concentrate on the projection of sample eigenvectors onto the corresponding population eigenvector. It's worth noting that our decision to concentrate on the simplified scenario is motivated by the desire to offer a clear and focused presentation of our main contributions. 

More precisely, let's assume that the population spiked eigenvalues of  $\bSi_n$, are denoted by $\alpha_1>\cdots>\alpha_K$. For the population eigenvector $\mb v_{k}$ of the $k$-th spiked eigenvalue $\alpha_k$, we denote its associated sample version as $\mathcal{V}_k$. Our interest lies in the inner product  $\mathcal{I}_k=\mb v_k^T\mathcal{V}_k$ of these two vectors. Assuming the separation condition that $\min_{j \neq k}\left|\frac{\alpha_k}{\alpha_j}-1\right|>d$, according to the Cauchy integral formula, we have the following equality
\begin{align*}
\mathcal{I}_k^2=-\frac{1}{2 \pi  {i}} \oint_{\zeta_k} \mb v_k^T \Upsilon\(\bS_n,z\) \mb v_k \mathrm{d} z,
\end{align*}
where $\zeta_k$ enclosing $\lambda_k$ but  excludes the other eigenvalues. Thus, we turn our attention to the study of $\mb v_k^T \Upsilon\(\bS_n,z\) \mb v_k.$
Noting that
\begin{align*}
	\mb v_k^T \Upsilon\(\bS_n,z\) \mb v_k
=&\be_k^T\left(\boldsymbol{\Lambda^{1/2}}\bU^T\bY_n\bY_n^T\bU\boldsymbol{\Lambda^{1/2}}-z \mathbf{I}\right)^{-1}\be_k,
	\end{align*}
by the equation
$$\mb B(\lambda\bI-\bA\bB)^{-1}=(\lambda\bI-\mb B\mb A)^{-1}\mb B,$$
we find the right hand side to be
\begin{align*}
	-\frac{1}{\alpha_k}\(\frac{z}{\alpha_k}+
{z}\bu_k^T\bY_n\(\bY_n^T\bSi_{1p}\bY_n-z\bI\)^{-1}\bY_n^T\bu_k\)^{-1}.
\end{align*}
Now, by applying the residue theorem and combining it with Lemma \ref{thmspike}, we establish the following result.
\begin{thm}
	Under assumptions in Lemma \ref{thmspike}, we have 
$$\mathcal{I}_k^2-\frac{\mathcal{G}_{2n}'(\alpha_k)}{\mathcal{G}_{2n}(\alpha_k)/\alpha_k} \to 0, a.s..$$	 
\end{thm}

We observe a notable phenomenon: the asymptotic properties of spiked sample eigenvectors remain consistent with the Gaussian case as long as $\nu_p=o(p^2)$. Under this condition, the almost sure limit of $\mathcal{I}_k^2$ is given by   $$\({1-c\int \frac{t^2}{(\alpha_k-t)^2}dH_{1}(t)}\)\({1+c\int \frac{t}{(\alpha_k-t)}dH_{1}(t)}\)^{-1},$$
which tends to zero as $\alpha_k\to 0$ and converges to a positive constant otherwise. This observation aligns with the understanding that, asymptotically, a bias angle will emerge between the true principal component and the estimated one in high dimensions unless the true principal component is divergent.

\begin{rem}
The above focused approach enables us to examine the angles between sample and population eigenvectors, shedding light on the asymptotic properties within various spiked model frameworks. By considering the alignment between these vectors, we gain insights into how the sample and population eigenvectors behave in the presence of a dominant signal, offering a nuanced understanding of their statistical properties. Certainly, exploring the more general case with multiple population spiked eigenvalues and the fluctuation of $\mathcal{I}_k$ is an intriguing avenue for extension. This extension may involve more complex but traceable calculations, and we leave it as a potential direction for future research. In light of Lemma \ref{thmspike} and the discussion above, we posit that the theoretical findings in \cite{Bao2022}, which investigate the fluctuations of $\mathcal{I}_k$ under ICS, apply to a wider spectrum of elliptical distributions by setting the fourth moment to $3$, provided that $\nu_p=o(p^2)$.
\end{rem}

We conclude this section by presenting numerical simulations to validate the correctness of our theoretical results.

Let's start with the simulation results for the spiked eigenvalues.
Consider two cases: $p=50, n=100$ and $p=200,n=400$. Set the corresponding population covariance matrix as $\bSi=\bU_0\bD_0\bU_0^*=8\bu_{0,1}\bu_{0,1}^T+\sum_{j=2}^{p}d_j\bu_{0,j}\bu_{0,j}^T,$
where $\bD={\rm Diag}\(8,d_2,\cdots,d_p\)$ with $d_j$'s i.i.d. chosen from $U(0,1)$ and $\bU_0=\(\bu_{0,1},\cdots,\bu_{0,p}\)$ is the eigenmatrix of the Toeplitz matrix $\bA=(a_{i,j})$ where $a_{i,j}=0.9^{|i-j|}$ for $i,j=1,\cdots,p.$ It can be observed that the population matrix has a spiked population eigenvalue of 8. For each pair of $p$ and $n$, we draw $n$ i.i.d. samples from an elliptical distribution with the given p-dimensional population covariance matrix $\bSi$. We consider four different types of elliptical distributions where $(a):\nu_p=p^2$, $(b):\nu_p=p$, $(c):\nu_p=p^{1/2}$ and $(d):\nu_p=0$. For each sample, we compute the largest sample eigenvalue and repeat this procedure 10,000 times.

The following figures depict the agreement between the empirical distribution of the sample spiked eigenvalues and Gaussian distributions. We also label the sample mean and sample variance under each situation.
\begin{figure}[htbp]
	\centerline{\includegraphics[height = 3.5 in,width = 5.5in]{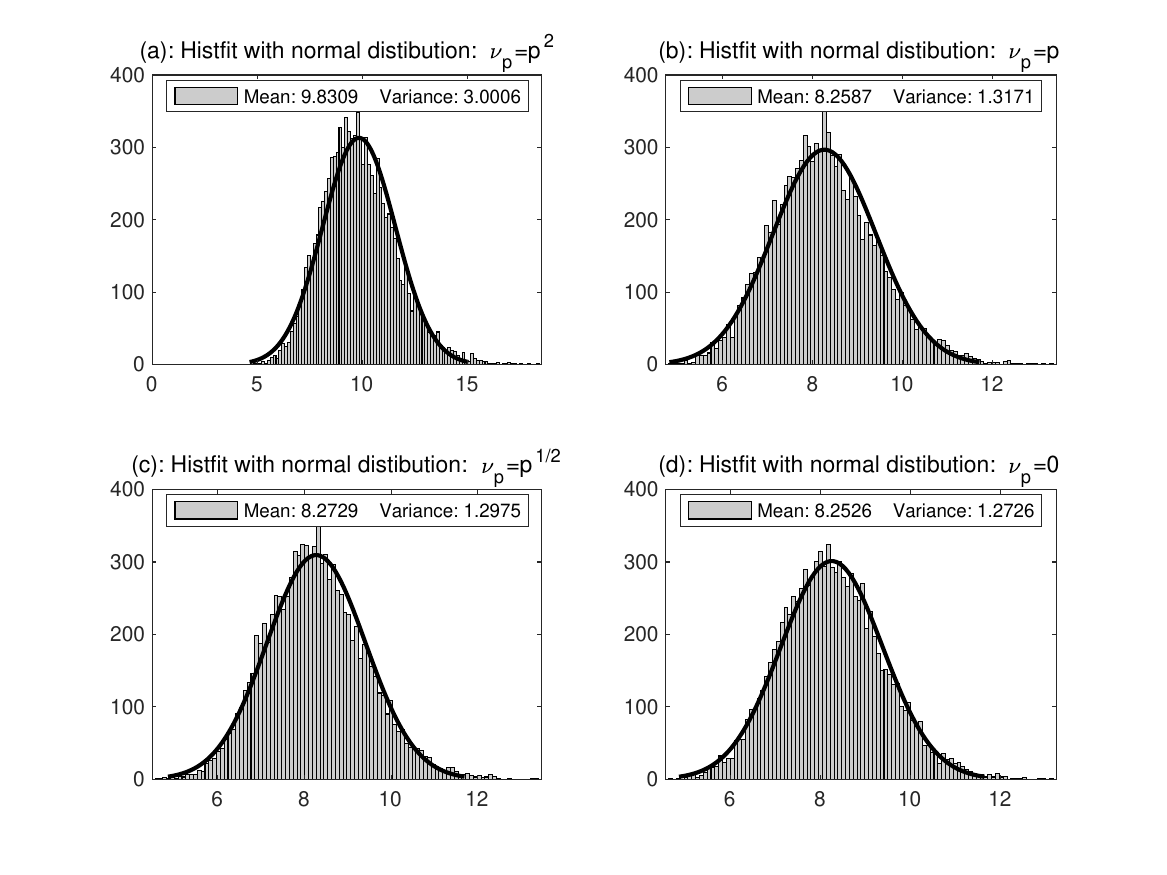}}	
\caption{Simulation results for the empirical distribution of spiked eigenvalues under different elliptical distributions. The dimension and sample size are $p=50$ and $n=100$. The x-axis represents the empirical sample spiked eigenvalues, while the y-axis represents the density. The black curve corresponds to Gaussian distribution with corresponding parameters given by the labels.}
\end{figure}

\begin{figure}[htbp]
	\centerline{\includegraphics[height = 3.5 in,width = 5.5in]{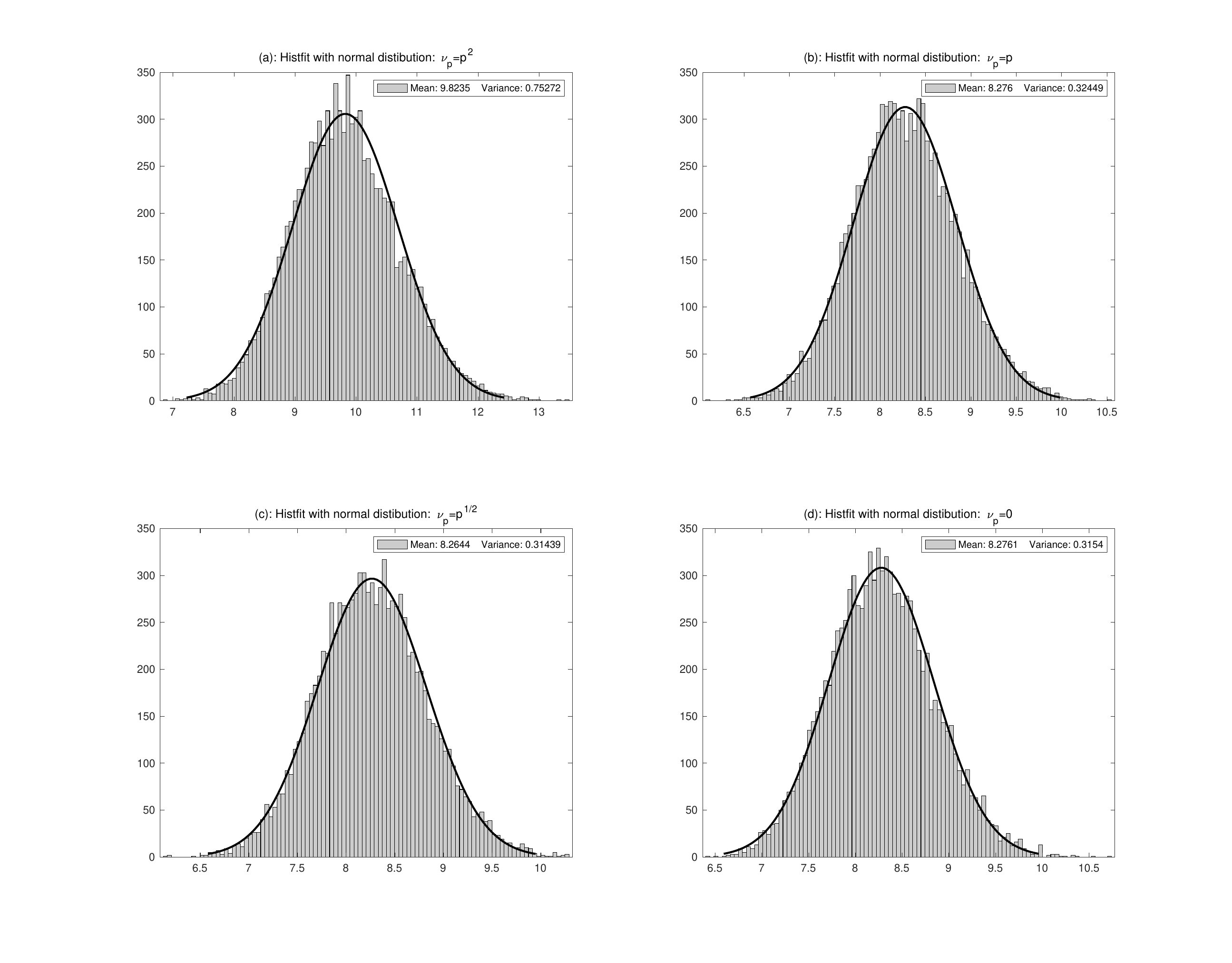}}	
\caption{Simulation results for the empirical distribution of spiked eigenvalues under different elliptical distributions. The dimension and sample size are $p=200$ and $n=400$. The x-axis represents the empirical sample spiked eigenvalues, while the y-axis represents the density. The black curve corresponds to Gaussian distribution with corresponding parameters given by the labels.}
\end{figure}
Two key observations can be drawn from Fig.1 and Fig.2.
Firstly, both figures demonstrate the good normality of empirical spiked sample eigenvalues under all cases. This suggests that the sample spiked eigenvalues exhibit properties akin to a normal distribution in various scenarios.
Secondly, a comparison between the two figures indicates that as $p\to\infty$, the asymptotic properties of sample spiked eigenvalues remain consistent as long as $\nu_p=o(p^2)$, aligning with our theoretical results. This transition in behavior is a noteworthy phenomenon in high-dimensional statistics.

In the subsequent analysis, we delve into simulations centered on spiked eigenvectors, maintaining the same population covariance matrix settings as in previous investigations. We specifically explore two scenarios with distinct dimension-to-sample size ratios: $c_{n,1} = p/n = 0.5$ and $c_{n,2} = p/n = 2$. Our exploration spans the range of $p$ from 64 to 512 in increments of 32. For each combination of $p$ and $n$, we draw $n$ samples from four distinct elliptical distributions, as previously considered. Subsequently, we compute the inner product of the population spiked eigenvector $\bu_{0,1}$ with its sample counterpart. This process is iterated 5000 times, and we compute the average under each distribution.

The resulting averages are then examined in terms of scatter plots against the varying values of $p$ for the different dimension-to-sample size ratios. Specifically, the scatter plots are presented in Fig.3 and Fig.4, providing a visual representation of how the average inner product behaves as the dimension $p$ varies.     
\begin{figure}[htbp]
	\centerline{\includegraphics[height = 2.5 in,width = 5.5in]{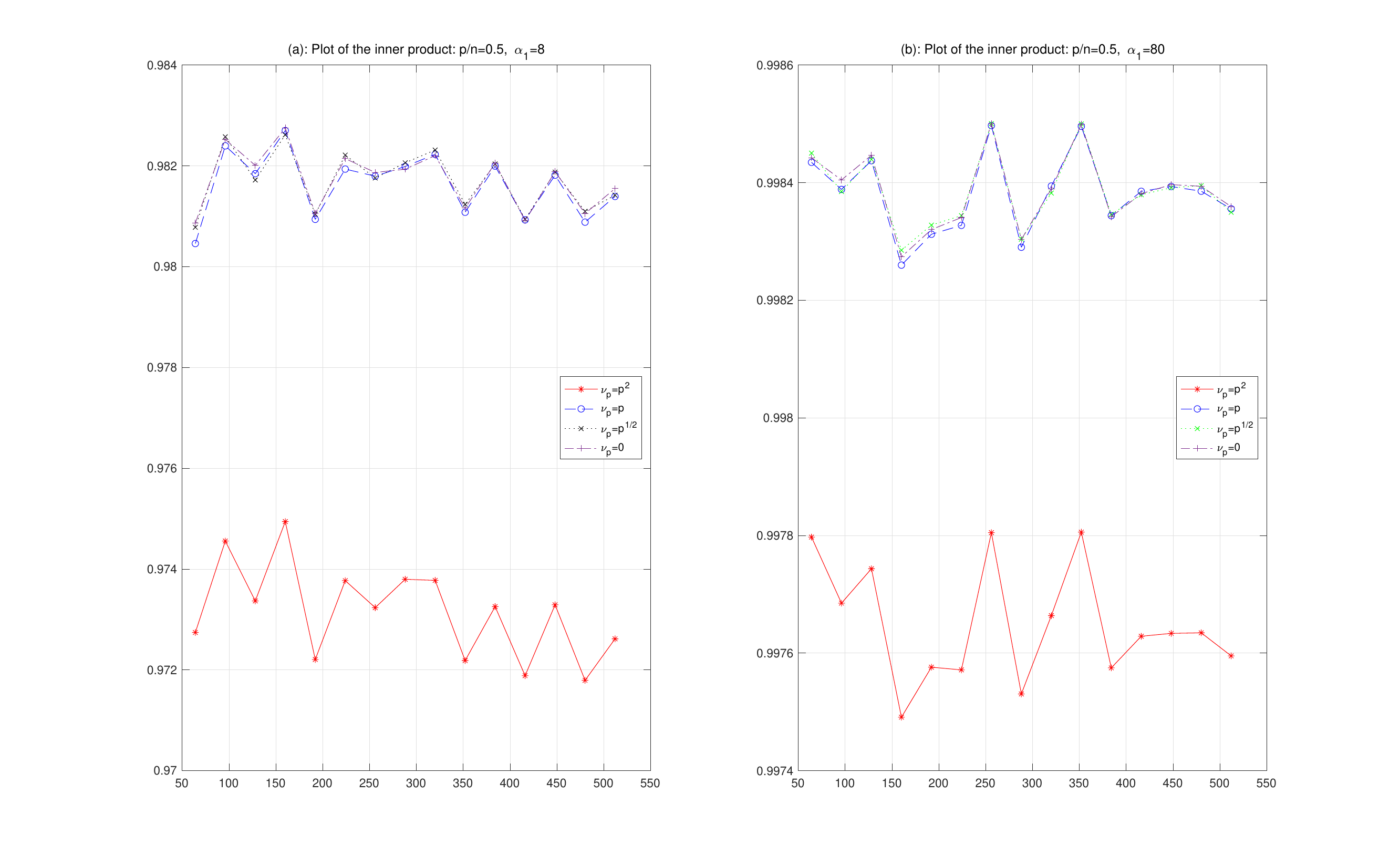}}	
\caption{Graphs depicting simulation results illustrate the averages of empirical inner products between the population's spiked eigenvector, denoted as $\mathbf{u}_{0,1}$, and its corresponding sample counterpart. The dimensionality $p$ ranges from 64 to 512 in increments of 32, maintaining a ratio of $p/n=0.5$. The x-axis denotes the dimension, while the y-axis represents the empirical inner product.}
\end{figure}

\begin{figure}[htbp]
	\centerline{\includegraphics[height = 2.5 in,width = 5.5in]{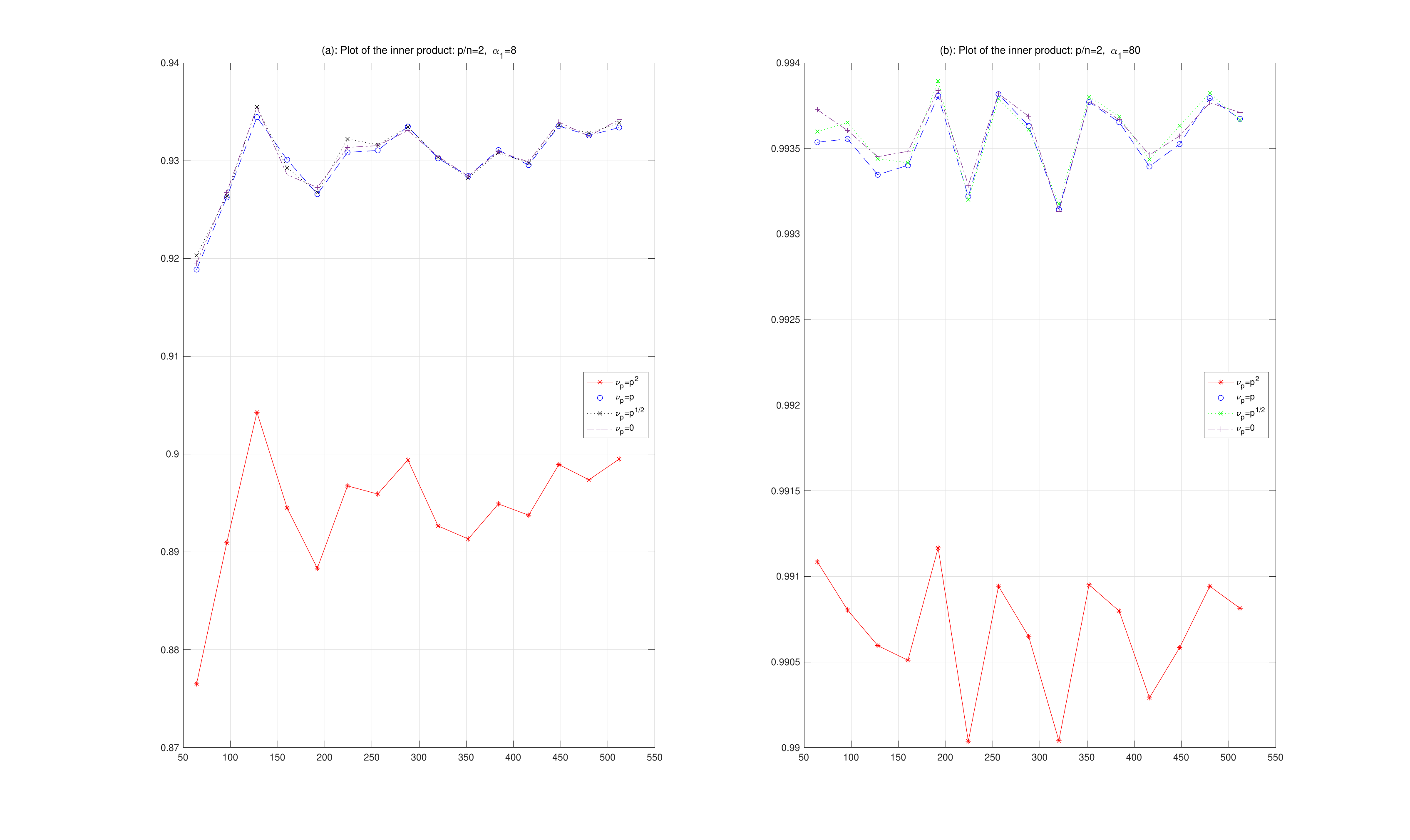}}	
\caption{Graphs depicting simulation results illustrate the averages of empirical inner products between the population's spiked eigenvector, denoted as $\mathbf{u}_{0,1}$, and its corresponding sample counterpart. The dimensionality $p$ ranges from 64 to 512 in increments of 32, maintaining a ratio of $p/n=2$. The x-axis denotes the dimension, while the y-axis represents the empirical inner product.}
\end{figure}

The figures reveal that, as the dimensionality $p$ increases, the asymptotic properties of the sample spiked eigenvector remain consistent, provided that $\nu_p = o(p^2)$, aligning with our theoretical results. However, when $\nu_p = p^2$, the inner product between the population's spiked eigenvector and its corresponding sample counterpart converges to a different value compared to the cases where $\nu_p = o(p^2)$. An interesting observation emerges: in our setting, a higher divergence rate of $\nu_p$ consistently leads to a larger angle between the population's spiked eigenvector and its corresponding sample counterpart.

\section{Concluding discussion} In this paper, we rigorously establish a joint CLT for bilinear forms, specifically those related to the resolvent of a random covariance matrix under ECS. Our analysis reveals a phase transition phenomenon, adding a nuanced dimension to our understanding of elliptical distribution.  Furthermore, we emphasize the practical importance of this CLT by showcasing its efficacy in exploring eigenvector statistics and the spiked model.  Through our investigation, we unveil consistent limiting properties for spiked eigenvalues and eigenvectors across a diverse set of elliptical distributions. This discovery underscores the robustness and adaptability of statistical tools originally designed for the spiked model under a Gaussian distribution.
For a detailed exposition of our results and proofs, we defer the reader to the appendix following our concluding discussion. Our work contributes to advancing the understanding of statistical properties in the context of random covariance matrices, particularly under ECS. 

\appendix
\section{Proof of the main theorems}
\subsection{Some definitions and preliminary results}
We initiate our proof by introducing crucial notations and some preliminary results. Let  $\br_k=\sqrt{c}_n\bGa_n{\bu}_k,$ and $\xi_k=m_p^{-1/4}\rho_k$. The matrices $\bA(z)$, $\bA_k(z)$, and $\bA_{k j}(z)$ are defined as follows:
$$
\bA(z)=\sum_{k=1}^{n}\xi_k^2\br_k\br_k^{T}-z\bI_p,\ \ \bA_k(z)=\bA(z)-\xi_k^2\br_k\br_k^{T},\ \ \bA_{k j}(z)=\bA_k(z)-\xi_j^2\br_j\br_j^{T}.
$$
Additionally, we introduce the matrices $\breve{\bA}_k(z)$ and $\breve{\bA}_{kj}(z)$, where $$\breve{\bA}_k(z)=
\sum\limits_{j<k}\xi_j^2\br_j\br_j^T+\sum\limits_{j>k}\breve\xi_j^2\breve{\br}_j\breve{\br}_j^{T}-z\bI_p, \ \breve{\bA}_{kj}(z)=\begin{cases}
	\breve{\bA}_k(z)-\xi_j^2\br_j\br_j^T,\quad j<k,\\
	\breve{\bA}_k(z)-\breve\xi_j^2\breve{\br}_j\breve{\br}_j^{T},\quad j>k.
\end{cases}$$
Here,$\breve\xi_k\breve{\br}_{k+1},\ldots,\breve\xi_n\breve{\br}_n$ are independent copies of $\xi_k{\br}_{k+1},\ldots,\xi_n{\br}_n$.
The conditional expectation given the samples $\xi_1\bx_1,\xi_2\bx_{2},\ldots,\xi_k\bx_k$ is denoted as $\E_k$. Moreover, we introduce
\begin{align}\label{nota}
\begin{split}
	\beta_k(z)=&\frac{1}{1+\xi_k^2\br_k^{T}\bA_k^{-1}(z)\br_k},\ b_k(z)=\frac{1}{1+\frac{\xi_k^2}n\rtr\(\bA_k^{-1}(z)\bSi_n\)}, \ 	\psi_k(z)=\frac{1}{1+{\xi_k^2} g_{1n}^0(z)}, 
	\\
	\beta_{kj}(z)=&\frac{1}{1+\xi_j^2\br_j^{T}\bA_{kj}^{-1}(z)\br_j}, \ \phi_k(z)=\frac{1}{1+\frac{\xi_k^2}n \E\rtr\(\bA^{-1}(z)\bSi_n\)},\ \breve{\psi}_{k}(z)=\frac{1}{1+{\breve\xi_k^2}g_{1n}^0(z)},
\end{split}
\end{align}
$${\gamma}_k(z)=\br_k^{T}\bA_k^{-1}(z)\br_k-\frac{1}{n}\rtr\(\bA_k^{-1}(z)\bSi_n\),\ {\eta}_k(z)=\br_k^{T}\bA_k^{-1}(z)\br_k-\frac{1}{n}\E\rtr\(\bA^{-1}(z)\bSi_n\),$$
$$\varepsilon_{k1}(z)=\br_{k}^{T} \mb A_{k}^{-1}(z) \pi_{n2} {\pi}_{n1}^{T}\mb A_{k}^{-1}(z) \br_k-\frac{1}{n}\pi_{n1}^{T} \mb A_{k}^{-1}(z) \bSi_n \mb A_{k}^{-1}(z) \pi_{n2},
$$
$$\varepsilon_{k2}(z)=\br_{k}^{T} \mb A_{k}^{-1}(z) \pi_{n4} {\pi}_{n3}^{T}\mb A_{k}^{-1}(z) \br_k-\frac{1}{n}\pi_{n3}^{T} \mb A_{k}^{-1}(z) \bSi_n \mb A_{k}^{-1}(z) \pi_{n4},
$$

The validity of the following inequality in the appropriate domain  $z\in \mathcal{Z}$ can be easily established as in \cite{BaiS04C,zhangzheng2022}:
\begin{equation}\label{b:bdd}
	\max(|b_k(z)|,|\phi_{k}(z)|, |\psi_{k}(z)|, |\breve\psi_{k}(z)|, |\beta_{kj}(z)|, |\beta_k(z)|)\leq C.
\end{equation}
Given that the support of $H_2$ is bounded, it follows that $\E\(\xi_1^{2q}\)\le C_q$. Employing the martingale difference decomposition method, Lemma \ref{quadratic}, and Lemma \ref{burk2}, the following inequality is obtained:
\begin{align}\label{decom1}
	&\E\left|\rtr\(\bA^{-1}(z)\bSi_n\)-\E\rtr\(\bA^{-1}(z)\bSi_n\)\right|^q\notag\\
=&\E\left|\sum_{k=1}^n\(\E_k-\E_{k-1}\)\beta_{k}(z)\xi_k^2\br_k^T\bA_k^{-1}(z)\bSi_n\bA_k^{-1}(z)\br_k\right|^q
\le C_qn^{q/2}.
\end{align}
This implies
\begin{align*}
\E\left|\eta_k(z)-\gamma_k(z)\right|^q\le	&\frac{C_q}{n^q}\E\left|\rtr\(\bA_k^{-1}(z)\bSi_n\)-\rtr\(\bA^{-1}(z)\bSi_n\)\right|^q\\
	&\qquad+\frac{C_q}{n^q}\E\left|\rtr\(\bA^{-1}(z)\bSi_n\)-\E\rtr\(\bA^{-1}(z)\bSi_n\)\right|^q\\
	=& \frac{C_q}{n^q}\E\left|\beta_{k}(z)\xi_k^2\br_k^T\bA_k^{-1}(z)\bSi_n\bA_k^{-1}(z)\br_k\right|^q+\frac{C_q}{n^{q/2}}
	\le{C_qn^{-q/2}}.
\end{align*}
Therefore,
\begin{align}\label{etab}
	\E\left|\eta_k(z)\right|^q
	\le&C_qn^{-q/2}.
\end{align}

\subsection{Proof of Theorem \ref{thmlsd}}
We proceed with the proof by establishing the almost sure convergence of the random part:
\begin{align*}
	\pi_{n1}^T\bA^{-1}(z)\pi_{n2}\xrightarrow{a.s.} m_{12}(z).
\end{align*}
This can be divided into two parts for comprehensive demonstration:
\begin{itemize}
	\item [(a)]: For the random part $\pi_{n1}^T\bA^{-1}(z)\pi_{n2}-\E\(\pi_{n1}^T\bA^{-1}(z)\pi_{n2}\)\xrightarrow{a.s.}0$,
	\item [(b)]: For the nonrandom part $\E\(\pi_{n1}^T\bA^{-1}(z)\pi_{n2}\)\rightarrow m_{12}(z)$.
\end{itemize}
\subsubsection{Almost sure convergence of the random part}\label{decomposition}
In this section, we aim to demonstrate the almost sure convergence of the random part:
\begin{align*}
\pi_{n1}^T\bA^{-1}(z)\pi_{n2}-\E\(\pi_{n1}^T\bA^{-1}(z)\pi_{n2}\)\xrightarrow{a.s.}0.
\end{align*}

Let $\E_0$ represent the unconditional expectation. Utilizing the inversion formula
\begin{align}\label{inver}
	\(\bA+\mb{\alpha\beta}^T\)^{-1}=\bA^{-1}-\frac1{1+\mb\beta^T\bA^{-1}\mb\alpha}\bA^{-1}\mb{\alpha\beta}^T\bA^{-1},
\end{align}
we obtain
\begin{align}\label{decom}
	\begin{split}
	\pi_{n1}^T\bA^{-1}(z)\pi_{n2}-&\E\(\pi_{n1}^T\bA^{-1}(z)\pi_{n2}\)
	=\sum_{k=1}^n\(\E_k-\E_{k-1}\)\pi_{n1}^T\(\bA^{-1}(z)-\bA_k^{-1}(z)\)\pi_{n2}\\
	=&-\sum_{k=1}^n\(\E_k-\E_{k-1}\)\beta_k(z)\xi_k^2\pi_{n1}^T\bA_k^{-1}(z)\br_k\br_k^T\bA_k^{-1}(z)\pi_{n2}.
	\end{split}
\end{align}
Notice that, from Lemma \ref{burk}, we have
\begin{align*}
	&\E\left|\sum_{k=1}^n\(\E_k-\E_{k-1}\)\beta_k(z)\xi_k^2\pi_{n1}^T\bA_k^{-1}(z)\br_k\br_k^T\bA_k^{-1}(z)\pi_{n2}\right|^4\\
	\le&C\E\(\sum_{k=1}^n\E_{k-1}\left|\pi_{n1}^T\bA_k^{-1}(z)\br_k\br_k^T\bA_k^{-1}(z)\pi_{n2}\right|^2\)^2\\
	&+C\delta_n^2p\sum_{k=1}^n\E\left|\pi_{n1}^T\bA_k^{-1}(z)\br_k\br_k^T\bA_k^{-1}(z)\pi_{n2}\right|^4\le \frac{C}{n^2}.
\end{align*}
This implies
$
	\pi_{n1}^T\bA^{-1}(z)\pi_{n2}-\E\(\pi_{n1}^T\bA^{-1}(z)\pi_{n2}\)\xrightarrow{a.s.}0.
$
\subsubsection{Convergence of $\E\(\pi_{n1}^T\bA^{-1}(z)\pi_{n2}\)$}

Denote $\mathbb{H}_n(z)=\E\(\xi_1^2\psi_1^2\)\bSi_n-z\bI_p=-z\(g_{2n}^0(z)\bSi_n+\bI_p\)$. Then
$
	\bA(z)-\mathbb{H}_n(z)=\sum_{k=1}^n\xi_k^2\br_k\br_k^T-\E\(\xi_1^2\psi_1^2\)\bSi_n.
$
Using \eqref{inver} and 
\begin{align}\label{beta2}
	\beta_k(z)=\phi_k(z)-\beta_k(z)\phi_k(z)\xi_k^2\eta_k(z),
\end{align}
we have
\begin{align}\label{expe}
	\begin{split}
	&\E\pi_{n1}^T\bA^{-1}(z)\pi_{n2}-\pi_{n1}^T\mathbb{H}_n(z)\pi_{n2}\\
=&-\E\pi_{n1}^T\mathbb{H}_n^{-1}(z)\(\sum_{k=1}^n\xi_k^2\br_k\br_k^T-\E\(\xi_1^2\psi_1^2\)\bSi_n\)\bA^{-1}(z)\pi_{n2}\\
=&-\frac1n\sum_{k=1}^n\E\xi_k^2\(\phi_k(z)-\psi_k(z)\) \pi_{n1}^T\mathbb{H}_n^{-1}(z)\bSi_n\E\bA_k^{-1}(z)\pi_{n2}\\
&+\sum_{k=1}^n\E\beta_k(z)\phi_k(z){\eta}_k(z)\xi_k^4\pi_{n1}^T\mathbb{H}_n^{-1}(z)\br_k\br_k^T\bA_k^{-1}(z)\pi_{n2}\\
&-\frac1n\sum_{k=1}^n\E\(\xi_1^2\psi_1^2\)\E\pi_{n1}^T\mathbb{H}_n^{-1}(z)\bSi_n\(\bA^{-1}_k(z)-\bA^{-1}(z)\)\pi_{n2}.\end{split}
\end{align}

Applying Lemma \ref{quadratic} and \eqref{etab}, we have
\begin{align*}
&\left|\E\pi_{n1}^T\bA^{-1}(z)\pi_{n2}-\pi_{n1}^T\(\E\(\xi_1^2\psi_1^2\)\bSi_n-z\bI_p\)^{-1}\pi_{n2}\right|\\
\le& C\left|\frac1n\E\rtr\(\bA^{-1}(z)\bSi_n\)-g_{1n}^0(z)\right|^2+C\sum_{k=1}^n\E^{1/2}\left|{\eta}_k(z)\right|^2+\frac Cn
=o(1).
\end{align*}
Using
$
	\E\(\xi_1^2\psi_1^2\)=\int\frac{x}{1+xg_{1n}^0(z)}dH_{2n}(x)=-zg_{2n}^0(z),
$ we obtain
\begin{align*}
	\E\pi_{n1}^T\bA^{-1}(z)\pi_{n2}+z^{-1}\pi_{n1}^T\(\bI_p+g_{2n}^0(z)\bSi_n\)^{-1}\pi_{n2}\to 0.
\end{align*}
This completes the proof of Theorem \ref{thmlsd}.

\subsection{The proof of Theorem \ref{thclt}}
In establishing the theorem, we adopt a methodology akin to the classical procedure developed in \cite{BaiM07A} under ICS. This involves undertaking a martingale difference decomposition followed by the application of the CLT for martingales. Notably, our approach draws inspiration from the work in \cite{BaiM07A,PanZ08C}, but with a significant departure as we streamline the intricate proofs substantially. This is achieved through the judicious use of the replacement of samples strategy. We posit that our simplified approach holds intrinsic interest in its own right. 

By virtue of the property of $\bu_j$, we rewrite $\bu_j=\by_j/\|\by_j\|$, $\bY_n=(\by_1,\cdots,\by_n)$,
\begin{align*}
	\bX_n=\bY_n\mathrm{diag}\left(\frac{\rho_1}{\|\by_1\|},\cdots,\frac{\rho_n}{\|\by_n\|}\right),\quad \bS_n=\frac1n\bGa_n\bY_n\mathrm{diag}\left(\frac{p\xi_1^2}{\|\by_1\|^2},\cdots,\frac{p\xi_n^2}{\|\by_n\|^2}\right)\bY_n^T\bGa_n^T
\end{align*}
where $\by_j\sim N({\bf 0}_p,\bI_p)$. 
Let $M_{n1}(z)=M_{n1}^1(z)+M_{n1}^2(z)$ and $M_{n2}(z)=M_{n2}^1(z)+M_{n2}^2(z)$, 
where
\begin{align*}
	M_{n1}^1(z)=&\sqrt p\left(\pi_{n1}^T\bA^{-1}(z)\pi_{n2}-\E\pi_{n1}^T\bA^{-1}(z)\pi_{n2}\right),\\
	M_{n2}^1(z)=&\sqrt p\left(\pi_{n3}^T\bA^{-1}(z)\pi_{n4}-\E\pi_{n3}^T\bA^{-1}(z)\pi_{n4}\right),\\
	M_{n1}^2(z)=&\sqrt p\left(\E\pi_{n1}^T\bA^{-1}(z)\pi_{n2}+z^{-1}\pi_{n1}^T
	(\bI_p+g_{2n}^0(z)\bSi_n)^{-1}\pi_{n2}\right),\\
	M_{n2}^2(z)=&\sqrt p\left(\E\pi_{n3}^T\bA^{-1}(z)\pi_{n4}+z^{-1}\pi_{n3}^T
	(\bI_p+g_{2n}^0(z)\bSi_n)^{-1}\pi_{n4}\right).
\end{align*}
Then the outline of the proof is as follows:
\begin{itemize}
	\item [(a)]: $\begin{pmatrix}
		M_{n1}^1(z),\ 
		M_{n2}^1(z)
	\end{pmatrix}^T$ converges weakly to a Gaussian process $M(z)$;
\item [(b)]: $\{M_{n1}(z)\}$ and $\{M_{n2}(z)\}$ both form a tight sequence on $\mathcal{Z}$;
\item [(c)]: $M_{n1}^2(z)$ and $M_{n2}^2(z)$ tend to zero for $z\in \mathcal{Z}$.
\end{itemize}
In the subsequent sections, we will systematically follow the outlined plan, proceeding step by step.

\subsubsection{Convergence in finite dimensions}
In this section, we aim to establish the convergence in distribution of the sum
\begin{align*}
\sum_{\ell=1}^2\sum_{j=1}^r\alpha_{j\ell}M_{n\ell}^1(z_j)
\end{align*}
for any positive integer $r$ and complex numbers $a_{j\ell}$, where $j=1,2$, and $\ell=1,\cdots,r$. This sum converges to a Gaussian random variable.

From \eqref{decom} and 
\begin{equation}\label{beta}
	\beta_k(z)=b_k(z)-\xi_k^2\beta_k(z)b_k(z){\gamma}_k(z),
\end{equation} 
it follows that
\begin{align*}
	&\sqrt p\left[\pi_{n1}^T\bA^{-1}(z)\pi_{n2}-\E\(\pi_{n1}^T\bA^{-1}(z)\pi_{n2}\)
	\right]\\
	=&-\sqrt p\sum_{k=1}^n\(\E_k-\E_{k-1}\)\xi_k^2b_k(z)\varepsilon_k(z)+\sqrt p\sum_{k=1}^n\(\E_k-\E_{k-1}\)\xi_k^4\beta_k(z)b_k(z){\gamma}_k(z)\varepsilon_k(z)\\
	&\qquad-\frac{\sqrt p}n\sum_{k=1}^n\(\E_k-\E_{k-1}\)\xi_k^2b_k(z)\pi_{n1}^T\bA_k^{-1}(z)\bSi_n\bA_k^{-1}(z)\pi_{n2}\\
	&\qquad+\frac{\sqrt p}n\sum_{k=1}^n\(\E_k-\E_{k-1}\)\xi_k^4\beta_k(z)b_k(z){\gamma}_k(z)\pi_{n1}^T\bA_k^{-1}(z)\bSi_n\bA_k^{-1}(z)\pi_{n2}.
\end{align*}
By computation and utilizing Lemma \ref{quadratic}, we obtain
\begin{align*}
	&\E\left|\frac{\sqrt p}n\sum_{k=1}^n\(\E_k-\E_{k-1}\)\xi_k^4\beta_k(z)b_k(z){\gamma}_k(z)\pi_{n1}^T\bA_k^{-1}(z)\bSi_n\bA_k^{-1}(z)\pi_{n2}\right|^2\\
\le&\frac{C}{n}\sum_{k=1}^n\E(\xi_k^8)\E\left|{\gamma}_k(z)\right|^2\le \frac{C}{n}.	
\end{align*}
and 
\begin{align*}
	&\E\left|\sqrt p\sum_{k=1}^n\(\E_k-\E_{k-1}\)\xi_k^4\beta_k(z)b_k(z){\gamma}_k(z)\varepsilon_k(z)\right|^2
\le Cp\sum_{k=1}^n\E^{1/2}\left|{\gamma}_k(z)\right|^4\E^{1/2}\left|\varepsilon_k(z)\right|^4
\le \frac{C}{n}.	
\end{align*}
Hence, we see
\begin{align*}
	&\sqrt p\left[\pi_{n1}^T\bA^{-1}(z)\pi_{n2}-\E\(\pi_{n1}^T\bA^{-1}(z)\pi_{n2}\)
	\right]
	=-\sqrt p\sum_{k=1}^n\(\E_k-\E_{k-1}\)\xi_k^2b_k(z)\varepsilon_{k1}(z)\\
	&\qquad-\frac{\sqrt p}n\sum_{k=1}^n\(\E_k-\E_{k-1}\)\xi_k^2b_k(z)\pi_{n1}^T\bA_k^{-1}(z)\bSi_n\bA_k^{-1}(z)\pi_{n2}+o_p(1).
\end{align*}
By \eqref{decom1} and \eqref{inver}, one finds
\begin{align}\label{al1}
	\E\left|b_k(z)-\psi_k(z)\right|^2
	\le&\frac C{n^2}\E\left|\xi_k^2\br_k^T\bA_k^{-1}(z)\bSi_n\bA_k^{-1}(z)\br_k\right|^2+o(1)=o(1), 
\end{align}
where used the fact that $n^{-1}\E\rtr\(\bA^{-1}(z)\bSi_n\)\to g_{1}(z)$. 
Therefore, we deduce from \eqref{al1}
\begin{align*}
	&\sqrt p\left[\pi_{n1}^T\bA^{-1}(z)\pi_{n2}-\E\(\pi_{n1}^T\bA^{-1}(z)\pi_{n2}\)
	\right]
	=-\sqrt p\sum_{k=1}^n\xi_k^2\psi_k(z)\E_k\varepsilon_{k1}(z)\\
	&\qquad-\frac{\sqrt p}n\sum_{k=1}^n\(\xi_k^2\psi_k(z)-\E\xi_k^2\psi_k(z)\)\E_k\pi_{n1}^T\bA_k^{-1}(z)\bSi_n\bA_k^{-1}(z)\pi_{n2}+o_p(1)\\
	&\quad\triangleq\sum_{k=1}^n Y_{k1}(z)+o_p(1).
\end{align*}
Applying the same procedure, it becomes evident
\begin{align*}
	&\sqrt p\left[\pi_{n3}^T\bA^{-1}(z)\pi_{n4}-\E\(\pi_{n3}^T\bA^{-1}(z)\pi_{n4}\)
	\right]
	=-\sqrt p\sum_{k=1}^n\xi_k^2\psi_k(z)\E_k\varepsilon_{k2}(z)\\
	&\qquad-\frac{\sqrt p}n\sum_{k=1}^n\(\xi_k^2\psi_k(z)-\E\xi_k^2\psi_k(z)\)\E_k\pi_{n3}^T\bA_k^{-1}(z)\bSi_n\bA_k^{-1}(z)\pi_{n4}+o_p(1)\\
	&\quad\triangleq\sum_{k=1}^n Y_{k2}(z)+o_p(1).
\end{align*}
Our next objective is to demonstrate that, for any positive integer $r > 0$, the sum
\begin{align*}
	\sum_{\ell=1}^2\sum_{j=1}^r\alpha_{j\ell}\sum_{k=1}^nY_{k\ell}(z_j)
\end{align*}
will converge in distribution to a Gaussian random variable. For any $$z_1,\ldots,z_r\in \mathbb C_+\quad \alpha_{11},\alpha_{12},\ldots,\alpha_{r1},\alpha_{r2}\in \mathbb R$$ and any $\varepsilon>0$, we have
\begin{align*}
	&\sum\limits_{k=1}^{n}\E\left(\left|\sum\limits_{{\ell}=1}^{2}\sum_{j=1}^r\alpha_{j\ell}Y_{k\ell}(z_j)\right|^2I\left(\left|\sum\limits_{{\ell}=1}^{2}\sum_{j=1}^r\alpha_{j\ell}Y_{k\ell}(z_j)\right|\geq\varepsilon\right)\right)\\
	\leq&\frac{ C}{\varepsilon^2}\sum\limits_{k=1}^{n}\sum\limits_{{\ell}=1}^{2}\sum_{j=1}^r\alpha_{j\ell}^4\E\left|Y_{k\ell}(z_j)\right|^4\to0
\end{align*}
where
$
	\E|Y_{k\ell}(z)|^4\le Cp\E\(\xi_k^8\)\E\left|\varepsilon_{k\ell}(z)\right|^4+\frac C{n^2}\E\(\xi_k^8\)\le \frac C{n^2}.
$
This implies the fulfillment of the Lindeberg condition for Lemma \ref{cltmar}.

 Then, we shall prove for $z_1,z_2\in\mathcal{Z}$,
\begin{align*}
&\sum_{k=1}^n\E_{k-1}\bigg[\bigg(\alpha_{11}Y_{k1}(z_1)+\alpha_{12}Y_{k2}(z_1)\bigg)\bigg(\alpha_{21}Y_{k1}(z_2)+\alpha_{22}Y_{k2}(z_2)\bigg)\bigg]\\
=&\alpha_{11}\alpha_{21}\sum_{k=1}^n\E_{k-1}\bigg(Y_{k1}(z_1)Y_{k1}(z_2)\bigg)+\alpha_{12}\alpha_{22}\sum_{k=1}^n\E_{k-1}\bigg(Y_{k2}(z_1)Y_{k2}(z_2)\bigg)\\
&+\alpha_{11}\alpha_{22}\sum_{k=1}^n\E_{k-1}\bigg(Y_{k1}(z_1)Y_{k2}(z_2)\bigg)+\alpha_{12}\alpha_{21}\sum_{k=1}^n\E_{k-1}\bigg(Y_{k2}(z_1)Y_{k1}(z_2)\bigg)
\end{align*}
tends to a constant in probability. We will now demonstrate the derivation of the limit for
$$ \sum_{k=1}^n\E_{k-1}\bigg(Y_{k1}(z_1)Y_{k2}(z_2)\bigg),$$ and the others follow a similar procedure.

To begin with, it is easy to get from Lemma \ref{quadeq} that
\begin{align*}
	&\sum_{k=1}^n\E_{k-1}\bigg(Y_{k1}(z_1)Y_{k2}(z_2)\bigg)\\
	=&\frac{c_nh_{n1}(z_1,z_2)}{n}\sum_{k=1}^n\E_k\(\pi_{n1}^T\bA_k^{-1}(z_1)\bSi_n\breve\bA_k^{-1}(z_2)\pi_{n4}\pi_{n3}^T\breve\bA_k^{-1}(z_2)\bSi_n\bA_k^{-1}(z_1)\pi_{n2}\)\\
	&+\frac{c_nh_{n1}(z_1,z_2)}{n}\sum_{k=1}^n\E_k\(\pi_{n1}^T\bA_k^{-1}(z_1)\bSi_n\breve\bA_k^{-1}(z_2)\pi_{n3}\pi_{n4}^T\breve\bA_k^{-1}(z_2)\bSi_n\bA_k^{-1}(z_1)\pi_{n2}\)\\
	&+\frac {ph_{n2}(z_1,z_2)}{n^2}\sum_{k=1}^n\E_k\pi_{n1}^T\bA_k^{-1}(z_1)\bSi_n\bA_k^{-1}(z_1)\pi_{n2}\E_k\pi_{n3}^T\bA_k^{-1}(z_2)\bSi_n\bA_k^{-1}(z_2)\pi_{n4}+o_p(1)\\
	=&\frac{c_nh_{n1}(z_1,z_2)}{n}\sum_{k=1}^n\E_k\(\pi_{n1}^T\bA_k^{-1}(z_1)\bSi_n\breve\bA_k^{-1}(z_2)\pi_{n4}\)\E_k\(\pi_{n3}^T\breve\bA_k^{-1}(z_2)\bSi_n\bA_k^{-1}(z_1)\pi_{n2}\)\\
	&+\frac{c_nh_{n1}(z_1,z_2)}{n}\sum_{k=1}^n\E_k\(\pi_{n1}^T\bA_k^{-1}(z_1)\bSi_n\breve\bA_k^{-1}(z_2)\pi_{n3}\)\E_k\(\pi_{n4}^T\breve\bA_k^{-1}(z_2)\bSi_n\bA_k^{-1}(z_1)\pi_{n2}\)\\
	&+\frac {c_nh_{n2}(z_1,z_2)}{n}\sum_{k=1}^n\E_k\(\pi_{n1}^T\bA_k^{-1}(z_1)\bSi_n\bA_k^{-1}(z_1)\pi_{n2}\)\E_k\(\pi_{n3}^T\bA_k^{-1}(z_2)\bSi_n\bA_k^{-1}(z_2)\pi_{n4}\)+o_p(1)\\
	\triangleq&\mathcal{I}_1+\mathcal{I}_2+\mathcal{I}_3+o_p(1)
\end{align*}
where the last equality is due to
\begin{align*}
	&\E\left|\pi_{n1}^T\(\bA_k^{-1}(z_1)-\E_k\bA_k^{-1}(z_1)\)\bSi_n\breve\bA_k^{-1}(z_2)\pi_{n4}\right|^2\\
=&\E\left|\sum_{j=k+1}^n\pi_{n1}^T\(\E_j-\E_{j-1}\)\(\bA_k^{-1}(z_1)-\bA_{kj}^{-1}(z_1)\)\bSi_n\breve\bA_k^{-1}(z_2)\pi_{n4}\right|^2
=O\(n^{-1}\)	,
\end{align*}
and
\begin{small}
\begin{align*}
h_{n1}(z_1,z_2)=&\int\frac{x^2dH_{2n}(x)}{(1+x g_{1n}^0(z_1))(1+x g_{1n}^0(z_2))}\to \frac{z_1g_2(z_1)-z_2g_2(z_2)}{g_1(z_1)-g_1(z_2)},\\
h_{n2}(z_1,z_2)=&h_{n1}(z_1,z_2)-\int\frac{xdH_{2n}(x)}{1+x g_{1n}^0(z_1)}\int\frac{xdH_{2n}(x)}{1+x g_{1n}^0(z_2)}\to z_1z_2\frac{\underline m(z_1)g_2(z_2)-\underline m(z_2)g_2(z_1)}{g_1(z_1)-g_1(z_2)}.
\end{align*}
\end{small}
Next, the matrix $\bA_k^{-1}(z)$ can be further decomposed as
\begin{equation}\label{Ajz}
	\bA_k^{-1}(z)=\mathbb{T}_n(z)+\bB_{k}(z)+\bC_{k}(z)+\bD_{k}(z)+\bF_k(z),
\end{equation}
where
$$\aligned
\mathbb{T}_n(z)&={-\left(z\bI-\frac{n-1}n\E\(\xi_1^2\psi_1(z)\)\cdot\bSi_n\right)^{-1}},\\
\bB_{k}(z)&=\sum\limits_{j\not=k}\xi_j^2\psi_j(z)
\mathbb{T}_n(z)
(\br_j\br_j^{T}-\frac1n\bSi_n)\bA_{kj}^{-1}(z),\\
\bC_{k}(z)&=\sum\limits_{j\not=k}(\beta_{kj}(z)-\psi_j(z))\xi_j^2
\mathbb{T}_n(z)
\br_j\br_j^{T}\bA_{kj}^{-1}(z),\\
\bD_{k}(z)&=\frac1n\sum\limits_{j\not=k}\(\xi_j^2\psi_j(z)-\E\xi_j^2\psi_j(z)\)
\mathbb{T}_n(z)
\bSi_n\bA_{kj}^{-1}(z),\quad\mbox{and}\\
\bF_k(z)&
=-\frac1n\E\(\xi_1^2\psi_1(z)\)\mathbb{T}_n(z)\sum\limits_{j\not=k}\beta_{jk}(z)\xi_j^2\bA_{kj}^{-1}(z)\br_j\br_j^T\bA_{kj}^{-1}(z).
\endaligned
$$
It is easy to obtain
\begin{align*}
&\E_k\(z_1\pi_{n1}^T\bA_k^{-1}(z_1)\pi_{n4}-z_2\pi_{n1}^T\breve\bA_k^{-1}(z_2)\pi_{n4}\)\\
=&- \pi_{n1}^T\(\bI_p+g_{2n}^0(z_1)\bSi_n\)^{-1}\pi_{n4}+\pi_{n1}^T\(\bI_p+g_{2n}^0(z_2)\bSi_n\)^{-1}\pi_{n4}+o_p(1)\\
=&\(g_{2n}^0(z_1)-g_{2n}^0(z_2)\) \pi_{n1}^T\(\bI_p+g_{2n}^0(z_1)\bSi_n\)^{-1}\bSi_n\(\bI_p+g_{2n}^0(z_2)\bSi_n\)^{-1}\pi_{n4}+o_p(1)\\
\xrightarrow{p}&\(g_{2}(z_1)-g_{2}(z_2)\) \lim_{n\to\infty}\pi_{n1}^T\(\bI_p+g_{2n}^0(z_1)\bSi_n\)^{-1}\bSi_n\(\bI_p+g_{2n}^0(z_2)\bSi_n\)^{-1}\pi_{n4}.
\end{align*}
On the other hand, rewrite $\E_k\(z_1\pi_{n1}^T\bA_k^{-1}(z_1)\pi_{n4}-z_2\pi_{n1}^T\breve\bA_k^{-1}(z_2)\pi_{n4}\)$ as
\begin{align*}
&\E_k\bigg[\pi_{n1}^T\bA_k^{-1}(z_1)\(z_1\breve\bA_k^{-1}(z_2)-z_2\bA_k^{-1}(z_1)\)\breve\bA_k^{-1}(z_2)\pi_{n4}\bigg]\\
=&\E_k\left[\pi_{n1}^T\bA_k^{-1}(z_1)\(\(z_1-z_2\)\sum_{j=1}^{k-1}\xi_j^2\br_j\br_j^T+\sum_{j=k+1}^{n}\(z_1\breve\xi_j^2\breve\br_j\breve\br_j^T-z_2\xi_j^2\br_j\br_j^T\)\)\breve\bA_k^{-1}(z_2)\pi_{n4}\right].
\end{align*}
We find the right hand side of the above equality equals to
\begin{align*}
&\(z_1-z_2\)\sum_{j=1}^{k-1}\E_k\left[\xi_j^2\psi_{j}(z_1)\psi_{j}(z_2)\pi_{n1}^T\bA_{kj}^{-1}(z_1)\br_j\br_j^T\breve\bA_{kj}^{-1}(z_2)\pi_{n4}\right]\\
&+z_1\sum_{j=k+1}^{n}\E_k\left[\breve\xi_j^2\breve\psi_{j}(z_2)\pi_{n1}^T\bA_k^{-1}(z_1)\breve\br_j\breve\br_j^T\breve\bA_{kj}^{-1}(z_2)\pi_{n4}\right]\\
&-z_2\sum_{j=k+1}^{n}\E_k\left[\xi_j^2\psi_{j}(z_1)\pi_{n1}^T\bA_{kj}^{-1}(z_1)\br_j\br_j^T\breve\bA_k^{-1}(z_2)\pi_{n4}\right]+o_p(1)\\
=&\(z_1-z_2\)\sum_{j=1}^{k-1}\xi_j^2\psi_{j}(z_1)\psi_{j}(z_2)\E_k\left[\pi_{n1}^T\bA_{kj}^{-1}(z_1)\(\br_j\br_j^T-\frac1n\bSi_n\)\breve\bA_{kj}^{-1}(z_2)\pi_{n4}\right]\\
&+\frac{\(z_1-z_2\)}n\sum_{j=1}^{k-1}\xi_j^2\psi_{j}(z_1)\psi_{j}(z_2)\E_k\left[\pi_{n1}^T\bA_{kj}^{-1}(z_1)\bSi_n\breve\bA_{kj}^{-1}(z_2)\pi_{n4}\right]\\
&+\frac{n-k}nw_n(z_1,z_2)\E_k\left[\pi_{n1}^T\bA_k^{-1}(z_1)\bSi_n\breve\bA_{k}^{-1}(z_2)\pi_{n4}\right]+o_p(1)\\
=
&\frac{\(z_1-z_2\)(k-1)}n\E\(\xi_1^2\psi_{1}(z_1)\psi_{1}(z_2)\)\E_k\left[\pi_{n1}^T\bA_{k}^{-1}(z_1)\bSi_n\breve\bA_{k}^{-1}(z_2)\pi_{n4}\right]\\
&+\frac{n-k}nw_n(z_1,z_2)\E_k\left[\pi_{n1}^T\bA_k^{-1}(z_1)\bSi_n\breve\bA_{k}^{-1}(z_2)\pi_{n4}\right]+o_p(1)\\
=
&w_{n}(z_1,z_2)\(1-\frac{k-1}{n}d_{n}(z_1,z_2)\)\E_k\left[\pi_{n1}^T\bA_k^{-1}(z_1)\bSi_n\breve\bA_{k}^{-1}(z_2)\pi_{n4}\right]+o_p(1),
\end{align*}
where the last second equality is from
\begin{align*}
	&\E\left|\sum_{j=1}^{k-1}\xi_j^2\psi_{j}(z_1)\psi_{j}(z_2)\E_k\left[\pi_{n1}^T\bA_{kj}^{-1}(z_1)\(\br_j\br_j^T-\frac1n\bSi_n\)\breve\bA_{kj}^{-1}(z_2)\pi_{n4}\right]\right|^2\\
	=&\E\(\xi_1^4\psi_{1}^2(z_1)\psi_{1}^2(z_2)\)\sum_{j=1}^{k-1}\E\left|\E_k\left[\pi_{n1}^T\bA_{kj}^{-1}(z_1)\(\br_j\br_j^T-\frac1n\bSi_n\)\breve\bA_{kj}^{-1}(z_2)\pi_{n4}\right]\right|^2\\
	&+\sum_{j\neq t}\E\Bigg\{\(\xi_j^2\xi_t^2\psi_{j}(z_1)\psi_{j}(z_2)\psi_{t}(z_1)\psi_{t}(z_2)\)\E_k\left[\pi_{n1}^T\bA_{kj}^{-1}(z_1)\(\br_j\br_j^T-\frac1n\bSi_n\)\breve\bA_{kj}^{-1}(z_2)\pi_{n4}\right]\\
	&\qquad\qquad\cdot\E_k\left[\pi_{n1}^T\bA_{kt}^{-1}(z_1)\(\br_t\br_t^T-\frac1n\bSi_n\)\breve\bA_{kt}^{-1}(z_2)\pi_{n4}\right]\Bigg\}
	=o(1)
\end{align*}
and
\begin{align*}
	&w_{n}(z_1,z_2)={z_1}\E\(\xi_1^2\psi_{1}(z_2)\)-{z_2}\E\(\xi_1^2\psi_{1}(z_1)\)\to z_1z_2\(g_2(z_1)-g_2(z_2)\),\\
	&d_{n}(z_1,z_2)=\frac{w_{n}(z_1,z_2)-\(z_1-z_2\)\E\(\xi_1^2\psi_{1}(z_1)\psi_{1}(z_2)\)}{w_{n}(z_1,z_2)}\to d(z_1,z_2).
\end{align*}
From these, one obtain
\begin{align*}
	&\E_k\left[\pi_{n1}^T\bA_k^{-1}(z_1)\bSi_n\breve\bA_{k}^{-1}(z_2)\pi_{n4}\right]\E_k\left[\pi_{n2}^T\bA_k^{-1}(z_1)\bSi_n\breve\bA_{k}^{-1}(z_2)\pi_{n3}\right]\\
=&\frac1{w_{n}^2(z_1,z_2)\(1-\frac{k-1}{n}d_{n}(z_1,z_2)\)^2}\E_k\(z_1\pi_{n1}^T\bA_k^{-1}(z_1)\pi_{n4}-z_2\pi_{n1}^T\breve\bA_k^{-1}(z_2)\pi_{n4}\)\\
&\qquad\qquad\cdot\E_k\(z_1\pi_{n2}^T\bA_k^{-1}(z_1)\pi_{n3}-z_2\pi_{n2}^T\breve\bA_k^{-1}(z_2)\pi_{n3}\)
+o_p(1),
\end{align*}
which yields
\begin{align}\label{sig1}
	\begin{split}
\mathcal{I}_1
	\xrightarrow{p}&h_{1}(z_1,z_2)\lim_{n\to\infty}\pi_{n1}^T\(\bI_p+g_{2n}^0(z_1)\bSi_n\)^{-1}\bSi_n\(\bI_p+g_{2n}^0(z_2)\bSi_n\)^{-1}\pi_{n4}\\
	&\quad\cdot\lim_{n\to\infty}\pi_{n2}^T\(\bI_p+g_{2n}^0(z_1)\bSi_n\)^{-1}\bSi_n\(\bI_p+g_{2n}^0(z_2)\bSi_n\)^{-1}\pi_{n3}\\
	=&h_{1}(z_1,z_2)r_{14}(z_1,z_2)r_{23}(z_1,z_2).
	\end{split}
\end{align}
Continuing with the same procedure, we deduce
\begin{align}\label{sig3}
	\begin{split}
	\mathcal{I}_2
	\xrightarrow{p}&h_{1}(z_1,z_2)r_{13}(z_1,z_2)r_{24}(z_1,z_2).
\end{split}
\end{align}

We are now prepared to address $\mathcal{I}_3$. It is known that 
\begin{align*}
		&\E_k\(\pi_{n1}^T\bA_k^{-1}(z_1)\pi_{n2}\)+z_1^{-1}\pi_{n1}^T\(\bI_p+g_{2n}^0(z_1)\bSi_n\)^{-1}\pi_{n2}
	\xrightarrow{p}0 \\ &z_1\E_k\(\pi_{n1}^T\bA_k^{-2}(z_1)\pi_{n2}\)+z_1\(z_1^{-1}\pi_{n1}^T\(\bI_p+g_{2n}^0(z_1)\bSi_n\)^{-1}\pi_{n2}\)'
	\xrightarrow{p}0 .
\end{align*}
Furthermore, it follows that
\begin{align*}
	&\E_k\(\pi_{n1}^T\bA_k^{-1}(z_1)\pi_{n2}\)\\
=&\sum_{j\neq k}\E_k\(\xi_j^2\pi_{n1}^T\bA_k^{-1}(z_1)\br_j\br_j^T\bA_k^{-1}(z_1)\pi_{n2}\)-z_1\E_k\(\pi_{n1}^T\bA_k^{-2}(z_1)\pi_{n2}\)\\
=&\sum_{j< k}\xi_j^2\psi_{j}^2(z_1)\E_k\(\pi_{n1}^T\bA_{kj}^{-1}(z_1)\(\br_j\br_j^T-\frac1n\bSi_n\)\bA_{kj}^{-1}(z_1)\pi_{n2}\)\\
&+\frac1n\sum_{j\neq k}\E_k\(\xi_j^2\psi_{j}^2(z_1)\pi_{n1}^T\bA_{kj}^{-1}(z_1)\bSi_n\bA_{kj}^{-1}(z_1)\pi_{n2}\)-z_1\E_k\(\pi_{n1}^T\bA_k^{-2}(z_1)\pi_{n2}\)+o_p(1)\\
=&\E\(\xi_1^2\psi_{1}^2(z_1)\)\E_k\(\pi_{n1}^T\bA_{k}^{-1}(z_1)\bSi_n\bA_{k}^{-1}(z_1)\pi_{n2}\)-z_1\E_k\(\pi_{n1}^T\bA_k^{-2}(z_1)\pi_{n2}\)+o_p(1).
\end{align*}
Consequently, we see
\begin{align*}
	\E_k\bigg(\pi_{n1}^T\bA_{k}^{-1}(z_1)&\bSi_n\bA_{k}^{-1}(z_1)\pi_{n2}\bigg)=\frac{\E_k\(\pi_{n1}^T\bA_k^{-1}(z_1)\pi_{n2}\)+z_1\E_k\(\pi_{n1}^T\bA_k^{-2}(z_1)\pi_{n2}\)}{\E\(\xi_1^2\psi_{1}^2(z_1)\)}+o_p(1)\\
	=&-\frac{[g_{2n}^0(z_1)]'}{z_1g_{2n}^0(z_1)}\pi_{n1}^T\(\bI_p+g_{2n}^0(z_1)\bSi_n\)^{-2}\bSi_n\pi_{n2}+o_p(1)\\
	\to&-\frac{g_2'(z_1)}{z_1g_2(z_1)}\lim_{n\to\infty}\pi_{n1}^T\(\bI_p+g_{2n}^0(z_1)\bSi_n\)^{-2}\bSi_n\pi_{n2}.
\end{align*}
This implies
\begin{align}\label{sig2}
	\mathcal{I}_3\xrightarrow{p}& {h_{2}(z_1,z_2)}r_{12}(z_1)r_{34}(z_2).
\end{align}
Combining \eqref{sig1}, \eqref{sig3}, and \eqref{sig2}, we conclude
\begin{align*}
	\sum_{k=1}^n\E_{k-1}\bigg(Y_{k1}(z_1)Y_{k2}(z_2)\bigg)
\xrightarrow{p}&h_{1}(z_1,z_2)r_{14}(z_1,z_2)r_{23}(z_1,z_2)+h_{1}(z_1,z_2)r_{13}(z_1,z_2)r_{24}(z_1,z_2)\\
&+{h_{2}(z_1,z_2)}r_{12}(z_1)r_{34}(z_2).
\end{align*}
\subsubsection{Tightness of $M_{nj}(z),j=1,2$}
We now proceed with the proof of tightness. Initially, owing to the similarity between $M_{n1}(z)$ and $M_{n2}(z)$, it suffices to demonstrate the tightness of the sequence of random functions $M_{n1}(z)$ for $z\in \mathcal{Z}$. Utilizing Theorem 12.3 of Billingsley and Section \ref{nonrand}, it is only necessary to show
\begin{align*}
	\sup_{n;z_1,z_2\in\mathcal{Z}}\frac{\E\left|M_{n1}^1(z_1)-M_{n1}^1(z_2)\right|^2}{\left|z_1-z_2\right|^2}\le C.
\end{align*}
Note that from \eqref{inver}
\begin{align*}
&\frac{M_{n1}^1(z_1)-M_{n1}^1(z_2)}{z_1-z_2}=\sqrt p\sum_{k=1}^n\left(\E_k-\E_{k-1}\right)\pi_{n1}^T\bA^{-1}(z_1)\bA^{-1}(z_2)\pi_{n2}\\
=
&-\sqrt p\sum_{k=1}^n\left(\E_k-\E_{k-1}\right)\beta_k(z_2)\xi_k^2\pi_{n1}^T\bA_k^{-1}(z_1)\bA_k^{-1}(z_2)\br_k\br_k^T\bA_k^{-1}(z_2)\pi_{n2}\\
&-\sqrt p\sum_{k=1}^n\left(\E_k-\E_{k-1}\right)\beta_k(z_1)\xi_k^2\pi_{n1}^T\bA_k^{-1}(z_1)\br_k\br_k^T\bA_k^{-1}(z_1)\bA_k^{-1}(z_2)\pi_{n2}\\
&+\sqrt p\sum_{k=1}^n\left(\E_k-\E_{k-1}\right)\beta_k(z_1)\beta_k(z_2)\xi_k^4\pi_{n1}^T\bA_k^{-1}(z_1)\br_k\br_k^T\bA_k^{-1}(z_1)\bA_k^{-1}(z_2)\br_k\br_k^T\bA_k^{-1}(z_2)\pi_{n2}\\
\triangleq&\mathcal{J}_1(z_1,z_2)+\mathcal{J}_2(z_1,z_2)+\mathcal{J}_3(z_1,z_2).
\end{align*}
Therefore, the ensuing steps aim to demonstrate
\begin{align}\label{tign}
	\sup_{n;z_1,z_2\in\mathcal{Z}}\E\left|\mathcal{J}_t(z_1,z_2)\right|^2\le C,\quad t=1,2,3.
\end{align}
Before proving \eqref{tign}, we provide moment bounds for specific random functions for $z\in\mathcal{Z}$ without delving into the details.  The first set of bounds pertains to any positive $q$ 
\begin{align}\label{boundcn1}
	\max\left\{\E\left\|\bA^{-1}(z)\right\|^q,\E\left\|\bA_k^{-1}(z)\right\|^q,\E\left\|\bA_{kj}^{-1}(z)\right\|^q\right\}\le C_q.
\end{align}
To save space and avoid redundancy, we omit this part and direct the reader to \cite{Bai2019} and \cite{zhangzheng2022} for more comprehensive details. The second set of bounds is given by:\begin{align}\label{bound2}
\E\left|\beta_k(z)\right|^q\le C_q,\quad{\rm and }\quad\E\left|\phi_k(z)\right|^q\le C_q.
\end{align}
Applying Lemma \ref{quadratic}, \eqref{boundcn1}, and \eqref{bound2}, it yields
\begin{align*}
	&\E\left|\mathcal{J}_1(z_1,z_2)\right|^2
\le p\sum_{k=1}^n\E\left|\beta_k(z_2)\xi_k^2\pi_{n1}^T\bA_k^{-1}(z_1)\bA_k^{-1}(z_2)\br_k\br_k^T\bA_k^{-1}(z_2)\pi_{n2}\right|^2\\
\le&p\sum_{k=1}^n\E^{1/2}\left|\beta_k(z_2)\right|^4\E^{1/2}\xi_k^8\E^{1/2}\left|\pi_{n1}^T\bA_k^{-1}(z_1)\bA_k^{-1}(z_2)\br_k\br_k^T\bA_k^{-1}(z_2)\pi_{n2}\right|^4\\
\le& \frac Cn\sum_{k=1}^n\E^{1/2}\left\|\bA_k^{-1}(z_1)\right\|^4\left\|\bA_k^{-1}(z_2)\right\|^8\le C.
\end{align*}
Using the same argument, we can derive
\begin{align*}
	\E\left|\mathcal{J}_2(z_1,z_2)\right|^2\le C\quad{\rm and}\quad\E\left|\mathcal{J}_3(z_1,z_2)\right|^2\le C.
\end{align*}
Therefore, we have completed the proof of tightness.

\subsubsection{Convergence of $M_{nj}^2(z),j=1,2$.}\label{nonrand}
A slight modification of the argument in \cite{Bai2019} allows us to extend their considered domain to $\mathcal{Z}$. Consequently, we establish that 
 $\sup_{n,z\in\mathcal{Z}}\|\mathbb{H}_n^{-1}(z)\|<\infty.$

Utilizing \eqref{expe} and Lemma \ref{quadratic}, we obtain
\begin{align}\label{expe1}
		M_{n1}^2(z)
		=&-{\sqrt p}\E\xi_1^2\(\phi_1(z)-\psi_1(z)\) \pi_{n1}^T\mathbb{H}_n^{-1}(z)\bSi_n\E\bA^{-1}(z)\pi_{n2}\notag\\
		&+\sqrt p\sum_{k=1}^n\E\beta_k(z)\phi_k(z){\eta}_k(z)\xi_k^4\pi_{n1}^T\mathbb{H}_n^{-1}(z)\br_k\br_k^T\bA_k^{-1}(z)\pi_{n2}+o(1)\\
		\triangleq&\mathcal{K}_1(z)+\mathcal K_2(z)+o(1).\notag
\end{align}
By utilizing \eqref{beta2} and Lemma \ref{quadeq}, it follows that
\begin{align*}
	\mathcal K_2(z)=& \sqrt p\sum_{k=1}^n\E\phi_k^2(z){\eta}_k(z)\xi_k^4\pi_{n1}^T\mathbb{H}_n^{-1}(z)\br_k\br_k^T\bA_k^{-1}(z)\pi_{n2}\\
	&-\sqrt p\sum_{k=1}^n\E\beta_k(z)\phi_k^2(z){\eta}_k^2(z)\xi_k^6\pi_{n1}^T\mathbb{H}_n^{-1}(z)\br_k\br_k^T\bA_k^{-1}(z)\pi_{n2}\\
	=
	&-\sqrt p\sum_{k=1}^n\E\beta_k(z)\phi_k^2(z){\eta}_k^2(z)\xi_k^6\pi_{n1}^T\mathbb{H}_n^{-1}(z)\br_k\br_k^T\bA_k^{-1}(z)\pi_{n2}+o(1).
\end{align*}
Combining the above equality and \eqref{etab}, one finds
\begin{align}\label{expe2}
	|\mathcal K_2(z)|\le& \frac C{\sqrt n}\sum_{k=1}^n\E^{1/2}|{\eta}_k(z)|^4\le \frac C{\sqrt n}.
\end{align}
Moreover, arguments from \cite{Bai2019} show that
\begin{align*}
	\sup_{n,z\in\mathcal{Z}}\sqrt p\(\frac1n\E\rtr\(\bA^{-1}(z)\bSi_n\)-g_{1n}^0(z)\)\to0.
\end{align*}
Thus, we have
\begin{align}\label{expe3}
	|\mathcal K_3(z)|\le& C\left|{\sqrt p}\E\xi_1^4\phi_1(z)\psi_1(z)\(\frac1n\E\rtr\(\bA^{-1}(z)\bSi_n\)-g_{1n}^0(z)\) \right|\to0.
\end{align}
Together with \eqref{expe1}-\eqref{expe3}, we conclude that for $z\in\mathcal{Z},$
$
	M_{n1}^2(z)\to0.
$
Additionally, employing the same procedure yields
\begin{align*}
	M_{n2}^2(z)\to0\quad z\in\mathcal{Z}.
\end{align*}

\subsection{Proof of Theorem \ref{vecClt}}
Note that, based on the discussion in Section \ref{FCLT}, following the approach of \cite{BaiM07A}, the proof of this lemma requires identifying a suitable domain $\mathcal{Z}$ and truncating the corresponding stochastic process. Then the desired result is indeed consequences of CLT for bilinear forms. To achieve this, let  $\mathcal{C}_n=\mathcal{C}\cap \{z:\Im(z)\ge n^{-1}\varepsilon_n\}$. The truncated process $\hat M_{n}(z)$ is defined as
$$\hat M_{n}(z)=\begin{cases}
	M_{n}(z), \qquad\qquad\qquad\qquad\ \ {\rm for}\ z\in\mathcal{C}_n\\
	M_{n}(x_r+{\rm sign}(\Im z)\cdot in^{-1}\varepsilon_n),\ {\rm for}\ x=x_r,v\in[0,n^{-1}\varepsilon_n],\\
	M_{n}(x_{\ell}+{\rm sign}(\Im z)\cdot in^{-1}\varepsilon_n),\ {\rm for}\ x=x_{\ell},v\in[0,n^{-1}\varepsilon_n].
\end{cases}$$ 
It can be verified with probability 1
\begin{align*} 
	M_{n}(z)-\hat M_{n}(z)\to 0, \quad{\rm for}\ z\in\mathcal{C}.
\end{align*}
Let $x_r$ be a number which is greater than the right endpoint of interval \eqref{inter} and let $x_{\ell}$ be a negative number if the left endpoint of interval \eqref{inter} is zero, otherwise let $$x_{\ell}\in\(0,a\liminf_n\lambda_{min}^{\bSi_n}I_{(0,1)}(c)\(1-\sqrt{c}\)^2\).$$
Let $v_0>0$ be arbitrary and $\mathcal{C}_u=\{u\pm iv_0:u\in[x_{\ell},x_r]\}$. Then
\begin{align}\label{contor}
	\mathcal{C}=\mathcal{C}_u\cup\big\{x_{\ell}+iv:v\in[-v_0,v_0]\big\}\cup\big\{x_{r}+iv:v\in[-v_0,v_0]\big\}.
\end{align}
Hence, we conclude that $\mathcal{C}$ is an appropriate domain, and the results follow directly from Theorem \ref{thclt}.

\section{Auxiliary lemmas}
We present the following lemmas, which are used in the proofs above.

\begin{lem}[]
\label{quadratic}
Let $\mathbf{A}=(a_{jk})$ be a $p\times p$ nonrandom matrix and $\mathbf{r}=\sqrt{c}_n\bGa_n\mathbf{u}$ where $\mathbf{u}\sim U(S^{p-1})$. Then for $q\ge2$,
\begin{align*}
\E\left|\mathbf{r}^T\mathbf{A}\mathbf{r}-\frac1n\operatorname{tr}(\mathbf{A}\bSi_n)\right|^q \le C_qp^{-q}r^{q/2}\|\mathbf{A}\bSi_n\|^q,
\end{align*}
where $r=\operatorname{rank}(\mathbf{A})$ and $C_q$ is a constant depending on $q$ only.
\end{lem}
\begin{proof}
From Lemma 5 in \cite{GaoH17H}, we have
\begin{align*}
\E\left|\mathbf{u}^T\bGa_n^T\mathbf{A}\bGa_n \mathbf{u}-\frac{1}{p} \operatorname{tr} (\mathbf{A}\bSi_n)\right|^q \le \frac{C_q}{p^q}\left[\operatorname{tr}(\mathbf{A}\bSi_n\mathbf{A}^T\bSi_n)^{q/2}+\operatorname{tr}(\mathbf{A}\bSi_n\mathbf{A}^T\bSi_n)^{q/2}\right],
\end{align*}
where $C_q$ is a positive constant depending only on $q$. Hence,
we obtain
\begin{align*}
\E\left|\mathbf{r}^T\mathbf{A}\mathbf{r}-\frac1n\operatorname{tr}(\mathbf{A}\bSi_n)\right|^q &\le C_qp^{-q}r^{q/2}\|\mathbf{A}\bSi_n\|^q.
\end{align*}
This completes the proof of the lemma.
\end{proof}

\begin{lem}
\label{quadeq}
Let $\mathbf{A}$ and $\mathbf{B}$ be two $p\times p$ nonrandom matrices, and $\mathbf{u}$ is uniformly distributed on the unit sphere $S^{p-1}$ in $\mathbb{R}^p$. Then we have
\begin{align*}
&\E\left(\mathbf{u}^T\mathbf{A} \mathbf{u}-\frac1p\operatorname{tr} \mathbf{A}\right)\left(\mathbf{u}^T\mathbf{B} \mathbf{u}-\frac1p\operatorname{tr} \mathbf{B}\right)= \frac{\operatorname{tr} (\mathbf{A}\mathbf{B}^T)+\operatorname{tr} (\mathbf{A}\mathbf{B})}{p(p+2)}-\frac{2\operatorname{tr}(\mathbf{A})\operatorname{tr}(\mathbf{B})}{p^2(p+2)}.
\end{align*}
\end{lem}

\begin{lem}[Theorem 35.12 of \cite{Billingsley95P}]
\label{cltmar}
Suppose that for each $n$, $Y_{n1}, Y_{n2}, \ldots, Y_{nr_n}$ is a real martingale difference sequence with respect to the increasing $\sigma$-field $\{\mathcal{F}_{nj}\}$ with second moments. If, as $n \to \infty$,
$
\sum_{j=1}^{r_n} \mathrm{E}\left(Y_{nj}^2 \,|\, \mathcal{F}_{n, j-1}\right) \stackrel{i.p.}{\longrightarrow} \sigma^2,
$
where $\sigma^2$ is a positive constant, and, for each $\varepsilon > 0,$
$
\sum_{j=1}^{r_n} \mathrm{E}\left(Y_{nj}^2 I_{\left(|Y_{nj}| \geq \varepsilon\right)}\right) \rightarrow 0,
$
then
$
\sum_{j=1}^{r_n} Y_{nr_n} \stackrel{\mathcal{D}}{\rightarrow} N\left(0, \sigma^2\right).
$
\end{lem}

\begin{lem}
\label{burk}
Let $\{X_k\}$ be a real martingale difference sequence with respect to the increasing $\sigma$-field $\mathcal{F}_k$, and let $\E_k$ denote conditional expectation with respect to $\mathcal{F}_k$. Then for $q \geq 2$,
$
\E\left|\sum X_k\right|^q \leq C_q \left[\E\left(\sum \E_{k-1}|X_k|^2\right)^{q/2} + \sum \E|X_k|^q\right].
$
\end{lem}

\begin{lem}
\label{burk2}
Let $\{X_k\}$ be a real martingale difference sequence with respect to the increasing $\sigma$-field $\mathcal{F}_k$, then for $q \geq 2$,
$
\E\left|\sum X_k\right|^q \leq C_q\E\left(\sum |X_k|^2\right)^{q/2}.
$
\end{lem}

\end{document}